\documentclass[12pt]{amsart}
\usepackage{hyperref}
\usepackage{amsmath,amssymb,amscd,amsthm,amsfonts}
\usepackage[all,cmtip]{xy}
\usepackage[numbers]{natbib}

\newtheorem{thm}{Theorem}[section]
\newtheorem{lemma}[thm]{Lemma}
\newtheorem{prop}[thm]{Proposition}
\newtheorem{cor}[thm]{Corollary}
\theoremstyle{definition}

\theoremstyle{remark}

\newtheorem{rmk}[thm]{Remark}
\numberwithin{equation}{section}

\newcommand{\BA}{{\mathbb {A}}}

\newcommand{\CA}{{\mathcal {A}}}

\newcommand{\cusp}{{\mathrm{cusp}}}

\newcommand{\disc}{{\mathrm{disc}}}

\newcommand{\GL}{{\mathrm{GL}}}

\newcommand{\Hom}{{\mathrm{Hom}}}

\newcommand{\Ind}{{\mathrm{Ind}}}

\newcommand{\Mp}{{\mathrm{Mp}}}

\newcommand{\rank}{{\mathrm{rank}}}
\renewcommand{\Re}{{\mathrm{Re}}}

\newcommand{\res}{{\mathrm{res}}}

\newcommand{\SL}{{\mathrm{SL}}}

\newcommand{\Span}{{\mathrm{Span}}}
\newcommand{\tr}{{\mathrm{tr}}}

\newcommand{\A}{\mathbb{A}}
\newcommand{\Z}{\mathbb{Z}} 
\newcommand{\R}{\mathbb{R}} \newcommand{\C}{\mathbb{C}}
 
\newcommand{\cS}{\mathcal{S}}\newcommand{\cH}{\mathcal{H}}
\newcommand{\cA}{\mathcal{A}}

\newcommand{\half}{\frac {1} {2}}

\newcommand{\til}{\widetilde}

\DeclareMathOperator{\Her}{\mathrm{Her}}

\newcommand{\isom}{\cong}

\DeclareMathOperator{\pdet}{\mathrm{pdet}}
\newcommand{\bN}{\mathbf{N}}

\DeclareMathOperator{\diag}{\mathrm{diag}}
\DeclareMathOperator{\lmod}{\backslash}
\let\Re\undefined
\DeclareMathOperator{\Re}{\mathrm{Re}}

\newcommand{\form}[2]{\langle{#1},{#2}\rangle}

\usepackage{xy}

\newcommand{\cP}{\mathcal {P}}
\newcommand{\cC}{\mathcal {C}}

\DeclareMathOperator{\FO}{\mathrm {FO}}
\DeclareMathOperator{\LO}{\mathrm {LO}}

\makeatletter

\newcommand{\Rmnum}[1]{\expandafter\@slowromancap\romannumeral #1@}
\makeatother

\begin{document}
\title[On $(\chi,b)$-factors of Cuspidal Forms of Unitary Groups \Rmnum{1}]
{On $(\chi,b)$-factors of Cuspidal Automorphic Representations of Unitary Groups \Rmnum{1}}

\author{Dihua Jiang}
\address{School of Mathematics\\
University of Minnesota\\
Minneapolis, MN 55455, USA}
\email{dhjiang@math.umn.edu}

\author{Chenyan Wu}
\address{School of Mathematics, University of Minnesota,\\ Minneapolis, MN 55455, USA}
\email{cywu@umn.edu}

\thanks{The research of the first named author is supported in part by the NSF Grants DMS--1301567.}
\date{\today}

\dedicatory{ Dedicated to Wen-Ching Winnie Li}

\begin{abstract}
Following the idea of \cite{MR2540878} for orthogonal groups, we introduce a new family of period integrals for cuspidal automorphic
representations $\sigma$ of unitary groups and investigate their relation with the occurrence of a simple global Arthur parameter $(\chi,b)$
in the global Arthur parameter $\psi_\sigma$ associated to $\sigma$, by the endoscopic classification of Arthur 
(\cite{MR3135650}, \cite{mok13:_endos}, \cite{KMSW}).
The argument uses the theory of theta correspondence. This can be viewed as a part of the $(\chi,b)$-theory outlined in \cite{jiang14:_autom_integ_trans_class_group_i} and can be regarded as a refinement of the
theory of theta correspondences and poles of certain $L$-functions, which was outlined in \cite{MR1159270}.
\end{abstract}
\maketitle{}

\section{Introduction}
\label{intro}

Let $F$ be a number field and $E$ be a quadratic extension of $F$. We denote by $\BA=\BA_F$ the ring of adeles of $F$, and by $\BA_E$ that of $E$.  
Let $G$ be a unitary group associated to an $m$-dimensional skew-Hermitian
vector space $X$. Consider $\sigma$ in $\CA_\cusp(G)$, the set of equivalence classes of irreducible cuspidal automorphic representations
of $G$ that occurs in the discrete spectrum, following the notation of \cite{MR3135650}. Let $\chi$ be an automorphic character of $\GL_1(\BA_E)$.
The tensor product $L$-functions $L(s,\sigma\times\chi)$ has been investigated through the work of Li in \cite{MR1166512} by using
the doubling method of Piatetski-Shapiro and Rallis, and the work of Langlands by calculating constant terms of the Eisenstein series with
cuspidal support $\chi\otimes\sigma$ (\cite{MR0419366} and \cite{MR2683009}). 
The philosophy of understanding the location of the poles of this family of $L$-functions in terms of
the theta correspondence for dual reductive pairs of unitary groups is outlined by Rallis in \cite{MR1159270}. A theory that
the location of the poles and non-vanishing of certain values of those $L$-functions detect the local-global principle for the first occurrences
the local and global theta correspondences has been completed through the work of Kudla-Rallis (\cite{MR1289491}) and that of Gan-Qiu-Takeda (\cite{gan14:_regul_siegel_weil_formul_secon}).

Periods of automorphic forms are also one of the important invariants attached to cuspidal automorphic representations in general. When
$G$ is an orthogonal group, Ginzburg, Soudry and the first named author of this paper introduce in \cite{MR2540878}
a family of period integrals for orthogonal
groups that detect the location of poles of the corresponding family of $L$-functions, and hence detect the first occurrence of
the theta correspondence for the dual reductive pairs of orthogonal groups and symplectic groups. The result and other relevant
results for orthogonal groups developed in \cite{MR2540878}, \cite{MR2330445}, \cite{MR1473165} and \cite{MR1473166} have recently found
their important applications in the investigation of Hodge type theorems for arithmetic manifolds in
a work of Bergeron, Millson and M\oe glin (\cite{bergeron12:_hodge}). This motivates the authors of this paper to look for a similar
theory when $G$ is a symplectic group (or metaplectic group) and unitary groups.
We expect that the results in this paper and in our forthcoming papers for unitary groups (\cite{jiang-wu-u2}) and symplectic groups (\cite{jiang-wu-sp}),
in addition to the work of the second named author (\cite{MR3071813}),
will play similar roles in the study of arithmetic manifolds associated to unitary groups and symplectic groups (\cite{bergeron14:_hodge} and \cite{MR3012154}).

The obvious adaption of period integrals for orthogonal groups in
\cite [eqs. (5.4), (6.1)] {MR2540878} to the unitary groups will not be enough for the theory for unitary groups, since the period integrals in \cite{MR2540878} 
require the target groups under the theta correspondence to be $F$-split (which is the case in loc. cit., since they are symplectic groups). 
Hence we have to figure out a new family of
period integrals through the direct calculation of the Fourier coefficients as suggested by the work of Li (\cite{MR1168122}), and
establish the theorems for symplectic groups (\cite{jiang-wu-sp}) and unitary groups (\cite{jiang-wu-u2}), which are parallel to those in \cite{MR2540878} for orthogonal groups. To complete the whole
story, we have to use the Arthur truncation to regularize the new family of period integrals for residues of the Eisenstein series with
cuspidal support $\chi\otimes\sigma$. As the new family of period integrals are more complicated in structure, 
the original proof in \cite{MR2540878} for orthogonal groups needs more details to
justify the convergence of some families of integrals from the truncation process. 
This is a main reason for us to write a second paper on this topic for unitary groups (\cite{jiang-wu-u2}), and
another paper for symplectic groups (\cite{jiang-wu-sp}).
In fact, the family of period integrals in the symplectic group case look even more complicated. The details will be given in \cite{jiang-wu-sp}.

It is well-known that theta correspondences, which extends the classical Shimura correspondence, is a powerful method to construct automorphic representations. One of the basic problems in the theory, dated in the early 1980s,
is to understand the relation between theta correspondences and the Langlands functorial transfers (\cite{MR658543}). 
The well-known conjecture of J. Adams
(\cite{MR1021501}) predicts the compatibility of theta correspondences with the Arthur-Langlands transfers. Hence, with the endoscopic
classification of Arthur, it is better to reformulate the theory of theta correspondences and related topics in automorphic forms 
in terms of the basic structure of the discrete spectrum of classical groups.
This issue was discussed with details in \cite[Section 7]{jiang14:_autom_integ_trans_class_group_i}.
We also refer to the work of M\oe glin (\cite{MR2906916}) on the relation between the Adams conjecture and theta correspondence.

Following \cite{MR3135650}, \cite{mok13:_endos}, \cite{KMSW}, and also \cite[Section 7]{jiang14:_autom_integ_trans_class_group_i},
the pair $(\chi, b)$ (for an integer $b\geq 1$) represents a simple global Arthur parameter for $G$.
The endoscopic classification of the discrete spectrum asserts that
each $\sigma$ in $\CA_\cusp(G)$ is attached to a global Arthur parameter $\psi=\psi_\sigma$.
As in \cite[Section 8.1]{jiang14:_autom_integ_trans_class_group_i}, it is easy to check that a simple global Arthur parameter $(\chi, b)$
occurs in the global Arthur parameter $\psi_\sigma$ with a maximal possible integer $b$ 
if and only if the partial $L$-function $L^S(s,\sigma\times\chi)$
is holomorphic for $\Re(s)>\frac{b+1}{2}$ and has a simple pole at $s=\frac{b+1}{2}$. Hence with the theory of endoscopic classification of
the discrete spectrum at hand, the program initiated by Rallis to understand the location of poles of the partial $L$-function $L^S(s,\sigma\times\chi)$ is eventually to detect the occurrence of the simple global Arthur parameter $(\chi, b)$ in the global Arthur parameter
$\psi_\sigma$ of $\sigma$. This is what the title of this paper indicates. The paper is organized as follows.

In Section~\ref{sec:pole-eis}, we relate  the location of poles of the family of Eisenstein series with cuspidal support $\chi\otimes\sigma$ to the location of poles of certain Siegel Eisenstein series. This is important for application to the global theta correspondences involved in the proofs of our main results. We also calculate the location of poles of this family of Eisenstein series in terms of  the location of poles of the partial $L$-function $L^S(s,\sigma\times\chi)$.
In Section~\ref{sec:fo-theta}, we recall the basic set-up of dual reductive pairs for unitary groups and define the notion of {\sl Lowest Occurrence} for the Witt towers of unitary groups. Note that the definition of the lowest occurrence in this paper is slightly different from that given in
\cite{MR2540878}, due to the different nature of the target Witt towers between orthogonal groups and unitary groups.
Theorem~\ref{thm:poleL} gives a bound for the lowest occurrence in terms of the location of poles of $L^S(s,\sigma\times\chi)$ or non-vanishing of $L^S(s,\sigma\times\chi)$ at $s=1/2$. This can be deduced from a more technical result (Theorem~\ref{thm:Eis-pole-LO}), which is proved in Section~\ref{sec:pf-thm}. A different approach using the Rallis inner product formula is undertaken in \cite{bergeron14:_hodge} to prove a similar result.
Based on the work of Li (\cite{MR1168122}) on non-existence of singular cusp forms for classical groups, we calculate explicitly the non-singular
Fourier coefficients of the first occurrence in Section~\ref{sec:period}.
This calculation produces a new family of  period integrals as stated in Propositions~\ref{prop:period-X->Y}
and \ref{prop:period-Y->X}. The other main result (Theorem~\ref{thm:mainperiod}), which is proved in Subsection~\ref{sec:mainperiod}, relates the existence of $(\chi,b)$-factor in the global Arthur parameter of $\sigma$ to the non-vanishing of the new family of period integrals on $\sigma$.
 At the end, we study another family of period
integrals which will be used in our further work on the topic.

Finally, we would like to thank the referee for helpful suggestions which improve the introduction of this paper.

\section{Poles of Certain Eisenstein series}
\label{sec:pole-eis}
\subsection{Notation and preliminary}
\label{sec:notation}
Let $E/F$ be a quadratic extension of number fields. Let $\A=\A_F$
be the ring of adeles of $F$ and $\A_E$ that of $E$. Let $\varepsilon_{E/F}$ be the quadratic character of $\A_F$ that is associated to the quadratic extension $E/F$ via Class Field Theory. We fix a
non-trivial additive character $\psi_F$ of $\A_F$ and define an additive character $\psi_E$ of $\A_E$ to be $\psi_F \circ (\half \tr_{E/F})$. We will usually write $\psi$ for $\psi_F$. It is, in general, clear from the context if $\psi$ denotes an Arthur parameter or an additive character. Fix $\delta\in
E^\times$ such that $\overline {\delta} = -\delta$. Scaling a Hermitian form by $\delta$  produces a skew-Hermitian form.
In this way we identify Hermitian and skew-Hermitian spaces. Let $X$ be a
vector space over $E$ of dimension $m$ endowed with the skew-Hermitian
form $\form{\ }{\
}_X$. 
Let $G (X)$ denote the unitary group of $X$. 
Let $\cH_a$ denote the split skew-Hermitian space of dimension $2a$. We use the dual basis $e_1^+,\ldots, e_a^+, e_1^-,\ldots,e_a^+$ for $\cH_a$ such that
\begin{equation*}
  \form{e_i^+}{e_j^-}_{\cH_a} =\delta_{ij} \quad\text { and} \quad \form{e_i^+}{e_j^+}_{\cH_a} = \form{e_i^-}{e_j^-}_{\cH_a}=0,
\end{equation*}
for $i,j =1,\ldots, a$, where $\delta_{ij}$ is the Kronecker delta.
Let $\ell_a^+$ (resp. $\ell_a^-$) be the span of $e_i^+$'s (resp. $e_i^-$'s). Then $\cH_a$ has the
 polarisation
\begin{equation*}
  \cH_a = \ell_a^+ \oplus \ell_a^-.
\end{equation*}
We form the dimension $m+2a$ skew-Hermitian space
\begin{equation*}
  X_a = X \perp \cH_a
\end{equation*}
and denote its isometry group  by $G (X_a)$. Let $Q_a$ be
the parabolic subgroup of $G (X_a)$ that stabilises $\ell_a^-$ and let $M_a$ be its Levi subgroup and $N_a$ its unipotent radical. Let $R_{E/F}$ denote the restriction of scalars of Weil. We
have
\begin{equation*}
  M_a \isom R_{E/F}\GL_a \times G (X).
\end{equation*}
Given $x\in R_{E/F}\GL_a$ and $h\in G (X)$ we denote the
corresponding element in $M_a$ by $m (x,h)$.

Fix a good maximal compact subgroup $K_a = \prod_v K_{a,v}$ of $G
(X_a) (\A)$ such that the Iwasawa decomposition
\begin{equation*}
  G (X_a) (\A) = Q_a (\A) K_a
\end{equation*}
holds.  For a Hermitian space $Y$, we may also form $Y_a$ with $\cH_a$ being a split Hermitian space.

Next we set up some notation for the doubling method and its extension.
Let $X'$ be the skew-Hermitian space with the same underlying space as $X$ but with the form scaled by $-1$:
\begin{equation*}
  \form{\ }{\ }_{X'}= -\form{\ }{\ }_{X}.
\end{equation*}
We will identify $G (X')$ with $G (X)$.
Form the  doubled space $W$ of $X$:
\begin{equation*}
  W = X \perp X'.
\end{equation*}
Define
\begin{equation*}
  X^\Delta = \{(x, x)\in W| x\in X\} \quad\text {and}\quad  X^\nabla = \{(x, -x)\in W| x\in X\} .
\end{equation*}
 Then
$  W = X^\Delta \oplus X^\nabla$
is a polarisation of $W$. Let $W_a = W \perp \cH_a$ and
\begin{equation*}
  \iota = \iota_1 \times \iota_2 : G (X_a)  \times G (X') \hookrightarrow G (W_a)
\end{equation*}
be the natural inclusion.
Define $X^\Delta_a = X^\Delta \oplus \ell_a^+$ and $X^\nabla_a =
X^\nabla \oplus \ell_a^-$. Then $W_a$ has a polarisation given by
\begin{equation*}
  W_a = X^\Delta_a \oplus X^\nabla_a.
\end{equation*}
Let $P_a$ denote the Siegel parabolic subgroup of $G (W_a)$ that stabilises $X^\nabla_a$.

\subsection{Certain Eisenstein series}
\label{sec:pole-eis-Qa}

We consider the Eisenstein series on $G (X_a)$ associated to the
parabolic subgroup $Q_a$. We will largely follow the notation in \cite{MR1361168}.

First we determine the Shahidi normalisation (\cite{MR2683009}). Since $Q_a$ is a maximal parabolic subgroup, $\mathfrak
{a}^*_{M_a}$ is one-dimensional. Let $T_0$ be the split component of the center of the minimal Levi subgroup of
$G (X_a)$, which is as fixed according to the set-up of the space $X_a$. 
Then $Q_a$ corresponds to a simple root $\alpha$ in the restricted
root system $\Delta (T_0, G (X_a))$. Let $\rho_{Q_a}$ be the half sum
of positive roots in $N_a$. Then the Shahidi normalisation identifies
$\C$ with $\mathfrak {a}^*_{M_a,\C}$ via $s\mapsto s\til {\alpha}$
where
\begin{equation*}
  \til {\alpha} = \form{\rho_{Q_a}}{\alpha}^{-1} \rho_{Q_a} = \left(\frac{m+a}{2}\right)^{-1}\rho_{Q_a}.
\end{equation*}
Here $\form{\ }{\ }$ is the usual Killing-Cartan form of the
non-restricted root system of $G (X_a)_{\overline {F}}$. As pointed out in \cite
[p. 10] {MR2683009}, one should first lift the relative roots to
absolute ones and then compute the Killing-Cartan form. The
value of the pairing is independent of the choice of lifts. Under this normalisation $\rho_{Q_a} = (m+a)/2$.

Let $H_a: M_a (\A) \rightarrow  \mathfrak {a}_{M_a}$ be the map defined
as follows: for $\chi \in \mathfrak {a}_{M_a}^*$,
\begin{equation*}
  \exp ( H_a (m) (\chi)) =  \prod_v | \chi (m) |_v.
\end{equation*}
Then
\begin{equation*}
  \exp ( H_a (m (x,h)) (s)) = |\det x|_{\A_E}^s
\end{equation*}
and
\begin{equation*}
  \exp ( H_a (m (x,h)) (\rho_{Q_a})) = |\det x|_{\A_E}^{\frac { m+a} {2}},
\end{equation*}
where $x\in R_{E/F}\GL_a (\A)$ and $h \in G (X) (\A)$.   We
extend $H_a$ to a function of $G (X_a) (\A)$ via the Iwasawa
decomposition.

Denote by $\CA_\cusp(G(X))$ the set of equivalence classes of irreducible
cuspidal automorphic representations of $G(X)(\BA)$.
Take $\chi$ to be a character of $E^\times \lmod \A_E^\times$ such
that $\chi|_{\A^\times} = \varepsilon_{E/F}^\epsilon$ for $\epsilon=0$
or $1$. Given such a $\chi$, we set $\epsilon_\chi$ to be the exponent
$\epsilon$. For $\sigma\in\CA_\cusp(G(X))$, denote by $\cA_{a} (s,\chi, \sigma) $ the space of smooth
$\C$-valued functions $\phi$ on $N_a (\A)M _a (F)\lmod G (X_a) (\A)$
that satisfy the following properties:
\begin{enumerate}
\item $\phi$ is right $K_a$-finite;
\item for any $x\in R_{E/F}\GL_a (\A)$ and $g\ \in G (X_a) (\A)$,
  \begin{equation*}
    \phi (m (x, I_X) g) = \chi (\det x) |\det x|_{\A_E}^{s+\rho_{Q_a}} \phi (g);
  \end{equation*}
\item for any fixed $k\in K_a$ and for any  $h \in G (X) (\A)$, the function
  \begin{equation*}
    h \mapsto \phi (m (I_a, h) k)
  \end{equation*}
is a smooth right $K_a\cap M_a^1$-finite vector in the space of $\sigma$.
\end{enumerate}

For $\phi \in \cA_{a}(0,\chi,\sigma)$ and $s\in \C$, define
\begin{equation*}
  \phi_s (g) := \exp ( H_a (g) (s)) \phi (g).
\end{equation*}
It is a section  in $\cA_{a} (s,\chi, \sigma) $ and we may identify it with a
smooth section in $\Ind_{Q_a (\A)}^{G (X_a) (\A)} \chi |\det
|_{\A_E}^s \otimes \sigma$.  Form the Eisenstein series
\begin{equation*}
  E^{Q_a} (g,s,\phi) := \sum_{\gamma \in Q_a (F) \lmod G (X_a) (F)} \phi_s (\gamma g).
\end{equation*}
By Langlands theory of Eisenstein series, it is absolutely convergent
for $\Re (s)> \rho_{Q_a}$ and has meromorphic continuation to
the whole $s$-plane, with finitely many poles in the half plane $\Re
(s) >0$, which are all real in our case (\cite{MR1361168}).

Let $\cP_a (\sigma,\chi)$ denote the set of positive poles of the
Eisenstein series $E^{Q_a} (g,s,\phi)$ for $\phi$ running over
$\cA_{a}(0,\chi,\sigma)$, i.e., $s_0 \in \cP_a (\sigma,\chi)$ if and
only if $s_0>0$ and there exists $\phi \in \cA_{a}(0,\chi,\sigma) $
such that $E^{Q_a} (g,s,\phi)$ has a pole at $s=s_0$. We have the
following proposition regarding the maximal pole in $\cP_1
(\sigma,\chi)$ and $\cP_a (\sigma,\chi)$. This is the unitary analog
to \cite [Remarque 1] {MR1473165} and \cite [Prop.~2.1] {MR2540878}.

\begin{prop}\label{prop:pole1->polea}
  Assume that $\cP_1 (\sigma,\chi)$ is non-empty and that $s_0$ is its maximum member. Then for all integers $a\ge1$, $s= s_0 +\frac{1}{2} (a-1)$ lies in $\cP_a (\sigma,\chi)$ and is its maximum member.
\end{prop}
\begin{proof}
 Let $Q_a'=M_a' N_a'$ be the parabolic subgroup of $G (X_a)$ that stabilises a full flag of isotropic subspaces
 \begin{equation*}
   \Span \{e_1^-\} \subset \Span \{e_1^-,e_2^-\} \subset\cdots \subset \Span \{e_1^-,\ldots,e_a^-\}
 \end{equation*}
 of $\ell_a^-$.  Thus its Levi part $M_a'$ is isomorphic to
 \begin{equation*}
   \underbrace  { R_{E/F}\GL_1 \times \cdots  \times R_{E/F}\GL_1}_{\text {$a$ fold}} \times G (X).
 \end{equation*}
Since $E^{Q_a} (g,s,\phi)$ is concentrated on $Q_a'$, to determine its  poles we only need to consider the poles of its constant term along $Q_a' $:
  \begin{equation}
\label{eq:const-term-eis}
E^{Q_a}_{Q_a'} (g,s,\phi) = \int_{[N_a']} E^{Q_a} (ng,s,\phi) dn.
  \end{equation}
To unfold we consider the double coset decomposition
\begin{equation*}
  Q_a (F)  \lmod G (X_a) (F) / Q_a' (F),
\end{equation*}
which is parameterised by the double cosets of Weyl groups
\begin{equation*}
  W_{Q_a'} \lmod W_{G (X_a)} / W_{Q_a}.
\end{equation*}
A system of  representatives for $W_{Q_a'} \lmod W_{G (X_a)} / W_{Q_a}$ is chosen to consist of those Weyl elements $w$ that have minimal length in their respective double coset $W_{Q_a'} w W_{Q_a}$.
Thus for $\Re (s)$ large, \eqref{eq:const-term-eis} is equal to
\begin{equation}
\label{eq:const-term-eis2}
  \sum_{w } \int_{ N_{a,w}' (\A)} \int_{[N_a' {}^ {,w}]}
 \sum_{m \in wQ_aw^{-1}\cap M_a' \lmod M_a'} \phi_s (w^{-1}mn_1n_2g) dn_1 dn_2,
\end{equation}
where the summation in $w$ runs over $W_{Q_a'} \lmod W_{G (X_a)} / W_{Q_a}$.
By cuspidality of $\sigma$, it can be seen that the terms corresponding to those $w$ such that $wQ_aw^{-1}\cap M_a' \lmod M_a'$ is not trivial must vanish.

We describe more precisely those contributing Weyl elements. Let $r$ denote the Witt index of $X$. Then the Weyl group $W_{G (X_a)}$ for $G (X_a)$ can be identified with the  signed permutations on $a+r$ letters. Let $T$ be a subset of $\{1,2,\ldots,a\}$. Order the elements of $T$ as $i_1<i_2<\cdots<i_t$ where $t$ is the cardinality of $T$ and the elements of the complement of $T$ as $j_1 < j_2 <\cdots < j_{a-t}$. Define $w_T$ by
\begin{align}\label{eq:w_T1}
  w_T (k) &= i_k \quad &\text {for $1 \le k\le t$ }\\
  \label{eq:w_T2}
  w_T (k) &= -j_{a-k+1} \quad &\text {for $t < k\le a$ }\\
  w_T (k) &=k \quad &\text {for $a < k \le a+r$ }. \nonumber
\end{align}
These are exactly the contributing Weyl elements.
Thus \eqref{eq:const-term-eis2} is equal to
\begin{align*}
  \sum_{T} \int_{ N_{a,w}' (\A)}   \phi_s (w_T^{-1}ng)  dn
\end{align*}
where $T$ runs over subsets of $\{1,2,\ldots,a\}$. Define $M (w_T,s)$ by
\begin{align}\label{eq:M_w_s}
  M (w_T,s)\phi_s (g) = \int_{ N_{a,w}' (\A)}   \phi_s (w_T^{-1}ng)  dn.
\end{align}
This is an intertwining operator for the parabolic subgroup $Q_a'$ and can be meromorphically continued to the whole $s$-plane. We note that $\phi_s$ can be viewed as a smooth section in
\begin{align*}
  \Ind_{Q_a' (\A)}^{G (X_a) (\A)} \chi |\ |_{\A_E}^{s-\frac{1}{2} (a-1)}\otimes
\chi |\ |_{\A_E}^{s-\frac{1}{2} (a-3)} \otimes \cdots \otimes
\chi |\ |_{\A_E}^{s-\frac{1}{2} (1-a)}\otimes \sigma.
\end{align*}
Decompose $w_T$ into ``simple reflections'' as in \cite [Lemma 4.2.1] {MR2683009}. There are two types of ``simple reflections'', one that only involves   subgroups of $G (X_a)$ isomorphic to $R_{E/F}\GL_2$ and one that involves the subgroups of the form $G (X_1)$.
The first type of the intertwining operators $M (w_{i,j},s)$ takes
\begin{equation*}
  \Ind_{ B (\A)}^{R_{E/F}\GL_2 (\A)} \chi|\ |_{\A_E}^{ s - \frac{1}{2} (a-1)+i-1} \otimes \chi|\ |_{\A_E}^{ - ( s - \frac{1}{2} (a-1)+j-1)}
\end{equation*}
to
\begin{equation*}
  \Ind_{ B (\A)}^{R_{E/F}\GL_2 (\A)} \chi|\ |_{\A_E}^{ - ( s - \frac{1}{2} (a-1)+j-1)} \otimes \chi|\ |_{\A_E}^{ s - \frac{1}{2} (a-1)+i-1}
\end{equation*}
for $t<j\le a$ and $j_{a-j+1} \le i \le t$ and the second type of the intertwining operators $M (w_{j},s)$ takes
\begin{equation*}
  \Ind_{ Q_1 (\A)}^{G (X_1) (\A)} \chi|\ |_{\A_E}^{  s - \frac{1}{2} (a-1)+j-1} \otimes \sigma
\end{equation*}
to
\begin{equation*}
   \Ind_{ Q_1 (\A)}^{G (X_1) (\A)} \chi|\ |_{\A_E}^{  - ( s - \frac{1}{2} (a-1)+j-1)} \otimes \sigma
\end{equation*}
for $t<j\le a$.

The intertwining operator $M (w_{i,j},s)$ is always holomorphic when $\Re (s)>\frac{1}{2} (a-1)$. It follows that the holomorphy of $M (w_T,s)$
for $\Re s >\frac{1}{2} (a-1)$ depends  only on the second type of the intertwining operators $M (w_{j},s)$.
Denote by $s_{\text{max}}$ the maximal pole in $\cP_1 (\sigma,\chi)$.
Assume that $s_0 \in \cP_a (\sigma,\chi)$ and $s_0 > \frac{1}{2} (a-1)$. Then for some $T$, $M (w_T,s)$ has a pole at $s=s_0$ and hence for some $j$, $M (w_j,s)$ has a pole at $s=s_0$. Since $M (w_j,s)$ only involves a group isomorphic to $G ( X_1)$, we get
\begin{align}\label{eq:upperbound-pole}
  s_0 - \frac{1}{2} (a-1)+j-1 \le s_{\text{max}}
\end{align}
and a fortiori
\begin{align*}
  s_0 - \frac{1}{2} (a-1) \le s_{\text{max}}.
\end{align*}

Now we show that $s_{\text{max}}+\frac{1}{2} (a-1) \in \cP_a (\sigma,\chi)$. From the computation above we only need to consider if  $M (w_\varnothing,s)$ has a pole at this position. All other $M (w_T,s)$'s are holomorphic at this position by assumption on the maximality of $s_{\text {max}}$. In addition, only the ``simple reflection'' $w_j$ for $j=0$ in the decomposition $w_\varnothing$ can possibly have a pole at $s=s_{\text{max}}+\frac{1}{2} (a-1)$. We try to single out this part as follows.

Let $Q_{1,a-1}$ be the standard parabolic subgroup of $G (X_a)$ that stabilises the flag $\Span \{e_1^-\} \subset \Span \{e_1^-,\ldots e_a^-\}$
and $Q_{a-1,1}$ the one that  stabilises the flag $\Span \{e_1^-,\ldots, e_{a-1}^-\} \subset \Span \{e_1^-,\ldots e_a^-\}$.
Let $w$ be the Weyl element defined by
\begin{align*}
  w (1)&=a;\\
  w (i)& = i-1 \quad \text {for $i \ne 1$.}
\end{align*}
We have an intertwining operator associated to $w_{\{ 1\}}$:
\begin{align}\label{eq:intertwining-w1}
  \Ind_{Q_{1,a-1} (\A)}^{G (X_a) (\A)} \chi|\ |_{\A_E}^{s_1}\otimes \chi|\det|_{\A_E}^{s_2} \otimes \sigma
\rightarrow
  \Ind_{Q_{1,a-1} (\A)}^{G (X_a) (\A)} \chi|\ |_{\A_E}^{s_1}\otimes\chi|\det|_{\A_E}^{-s_2} \otimes \sigma
\end{align}
and an intertwining operator associated to $w$:
\begin{align}\label{eq:intertwining-w}
  \Ind_{Q_{1,a-1} (\A)}^{G (X_a) (\A)} \chi|\ |_{\A_E}^{s_1}\otimes \chi|\det|_{\A_E}^{-s_2} \otimes \sigma
\rightarrow
  \Ind_{Q_{a-1,1} (\A)}^{G (X_a) (\A)} \chi|\det |_{\A_E}^{-s_2}\otimes \chi|\ |_{\A_E}^{s_1} \otimes \sigma.
\end{align}
It clear that there is the following canonical inclusion
\begin{align}\label{eq:inclusion}
  \Ind_{Q_a (\A)}^{G (X_a) (\A)} \chi |\det |_{\A_E}^s \otimes \sigma \hookrightarrow
\Ind_{Q_{1,a-1} (\A)}^{G (X_a) (\A)} \chi|\ |_{\A_E}^{s - \frac{1}{2} (a-1)}\otimes \chi|\det|_{\A_E}^{s+\frac{1}{2}} \otimes \sigma.
\end{align}
We specialise $s_1$ and $s_2$  to $s - \frac{1}{2} (a-1)$ and $s+\frac{1}{2}$ respectively and
 deduce that near $s=s_{\text{max}}+\frac{1}{2} (a-1)$, the intertwining operator \eqref{eq:intertwining-w1} is a bijection  by \eqref{eq:upperbound-pole} and the intertwining operator \eqref{eq:intertwining-w} is also a bijection by computation in several $R_{E/F}\GL_2$'s.

Consider the restriction map from $G (X_a)$ to $G (Ee_a^+ \oplus X\oplus Ee_a^-)$:
\begin{align}\label{eq:restriction}
  \Ind_{Q_{a-1,1} (\A)}^{G (X_a) (\A)} \chi|\det |_{\A_E}^{- (s+\frac{1}{2})}\otimes \chi|\ |_{\A_E}^{s - \frac{1}{2} (a-1)} \otimes \sigma
\rightarrow
\Ind_{Q_{1} (\A)}^{G (X_1) (\A)}  \chi|\ |_{\A_E}^{s - \frac{1}{2} (a-1)} \otimes \sigma.
\end{align}
The image of $\Ind_{Q_a (\A)}^{G (X_a) (\A)} \chi |\det |_{\A_E}^s \otimes \sigma$ under this restriction map in $\Ind_{Q_{1} (\A)}^{G (X_1) (\A)}  \chi|\ |_{\A_E}^{s - \frac{1}{2} (a-1)} \otimes \sigma$ is dense. Consider now the two intertwining operators associated to $w_\varnothing$ for $G (X_1)$ and $G (X_a)$ respectively. Then we have a commutative diagram:
\begin{equation*}
  \xymatrix {
\Ind_{Q_a (\A)}^{G (X_a) (\A)} \chi |\det |_{\A_E}^s \otimes \sigma
\ar [r]  \ar [d]^{M (w_\varnothing,s)}
& \Ind_{Q_{1} (\A)}^{G (X_1) (\A)}  \chi|\ |_{\A_E}^{s - \frac{1}{2} (a-1)} \otimes \sigma
\ar [d]^{M (w_\varnothing,s- \frac{1}{2} (a-1))}\\
\Ind_{Q_a (\A)}^{G (X_a) (\A)} \chi |\det |_{\A_E}^s \otimes \sigma \ar [r] &
\Ind_{Q_{1} (\A)}^{G (X_1) (\A)}  \chi|\ |_{\A_E}^{s - \frac{1}{2} (a-1)} \otimes \sigma
},
\end{equation*}
where the top horizontal arrow is the composed map of  \eqref{eq:inclusion}, \eqref{eq:intertwining-w1}  \eqref{eq:intertwining-w} and \eqref{eq:restriction} and the bottom horizontal arrow is the composition of an inclusion map with \eqref{eq:restriction}.
Since we have assumed that the right vertical arrow produces a pole at $s=s_{\text {max}} + \frac{1}{2} (a-1)$, the left vertical arrow must also produce a pole at $s=s_{\text {max}} + \frac{1}{2} (a-1)$. Thus we have shown $s=s_{\text {max}} + \frac{1}{2} (a-1)\in \cP_a (\sigma,\chi)$.
\end{proof}
Now we relate the poles of the partial $L$-function $L^S (s,\sigma\times \chi)$ to those of Eisenstein series.
\begin{prop} \label{prop:L-pole->Eis-pole}
Assume one of the followings.
  \begin{itemize}
  \item The partial $L$-function $L^S (s,\sigma\times \chi)$  has a pole at $s= s_0 > 1/2$  and is holomorphic for $\Re ( s) > s_0$.
\item  The partial $L$-function $L^S (s,\sigma\times \chi)$ is non-vanishing at $s=s_0=1/2$ and is holomorphic for $\Re (s) > 1/2$. In addition $m-\epsilon_\chi$ is even.
  \end{itemize}
  Then for all integers $a\ge 1$,  $s=s_0+\frac{1}{2} (a-1) \in \cP_a (\sigma,\chi)$.
\end{prop}
\begin{proof}
  Let $S$ be a finite set of places of $F$ such that for $v$ away from $S$, $v\nmid 2, \infty$, $E/F$, $G (X_a)$, $\chi$ and $\sigma$ are unramified. We focus on the unramified computation for $M (w_\varnothing,s)$ which is defined in the proof of Prop.~\ref{prop:pole1->polea}.

By the Gindikin-Karpelevich formula or more concretely by \eqref{eq:intertwining-w1} and \eqref{eq:intertwining-w} we find that the $c$-function is given by
\begin{align*}
  &\prod_{1\le j\le a}\frac{L^S (s-\frac{1}{2} (a-1)+j-1,\sigma\times\chi)}
  {L^S (s-\frac{1}{2} (a-1)+j,\sigma\times\chi)}\\
  \cdot&\prod_{1 \le i<j\le a} \frac{L_E ^S(2s- (a-1)+i+j-2)}{L_E ^S(2s- (a-1)+i+j-1)}\\
\cdot&
\begin{cases}
  \prod_{1\le j\le a}\frac {L_F^S (2s- (a-1)+2j-2,\varepsilon_{E/F}\chi)}
  {L_F^S (2s- (a-1)+2j-1,\varepsilon_{E/F}\chi)} & \text {for $m$ odd;}\\
\prod_{1\le j\le a}\frac {L_F^S (2s- (a-1)+2j-2,\chi)}
  {L_F^S (2s- (a-1)+2j-1,\chi)} & \text {for $m$ even.}
\end{cases}
\end{align*}
Here $L^S_F$ is the partial Tate $L$-function for $F$ away from $S$ and $L^S_E$ is the partial Tate $L$-function for $E$ away from places of $E$ lying above those in $S$.
This can be simplified to $I_1\cdot I_2 \cdot I_3$ where
\begin{align*}
  I_1 &=  \frac{L^S (s-\frac{1}{2} (a-1),\sigma\times\chi)}
  {L^S (s+\frac{1}{2} (a+1),\sigma\times\chi)}\\
  I_2 &= \prod_{2\le j \le a}\frac{L_E^S (2s- (a-1)+j-1)}{L_E^S (2s -  (a-1)+2j-2)}\\
I_3 &=
\begin{cases}
  \prod_{1\le j\le a}\frac {L_F^S (2s- (a-1)+2j-2,\varepsilon_{E/F}\chi)}
  {L_F^S (2s- (a-1)+2j-1,\varepsilon_{E/F}\chi)} & \text {for $m$ odd;}\\
\prod_{1\le j\le a}\frac {L_F^S (2s- (a-1)+2j-2,\chi)}
  {L_F^S (2s- (a-1)+2j-1,\chi)} & \text {for $m$ even.}
\end{cases}
\end{align*}

Assume the first condition. At $s=s_0 + \frac{1}{2} (a-1)$, the denominators  cannot have a pole and the numerators cannot have a zero.
The rest of the intertwining operators will not produce poles at $s=s_0 + \frac{1}{2} (a-1)$ since $s_0$ is `maximal'.
Thus the term corresponding to $M (w_\varnothing,s)$ is the only term that can contribute a pole at $s=s_0 + \frac{1}{2} (a-1) $ to
$E^{Q_a}_{Q_a'} (g,s,\phi)$. Therefore, $s_0 + \frac{1}{2} (a-1)$ belongs to $\cP_a (\sigma,\chi)$.

Assume the second condition. Then $I_1$ and $I_2$  are holomorphic and non-vanishing at $s=\frac {1} {2} + \frac{1}{2} (a-1)$ whereas the numerators of $I_3$  produce a pole and the denominators of $I_3$ are holomorphic at $s=\frac{1}{2}$. Again the rest of the intertwining operators will not produce poles at $s=\frac {1} {2} + \frac{1}{2} (a-1)$. Thus $\frac {1} {2} + \frac{1}{2} (a-1)$ belongs to $\cP_a (\sigma,\chi)$.
\end{proof}

Now we recall the following relation between a Siegel Eisenstein series on the doubled group $G (W_a)$ and $E^{Q_a}$ from \cite [Prop.~5.6, Lemma~5.8] {MR3071813}. We use $\sigma\otimes\chi^{-1}$ instead of $\sigma$ in the statement here to make $\Phi_{ f,s}(g_a)$ lie in $\cA_{a} (s,\chi, \sigma)$. A minor typo in the statement of the absolute convergence range is corrected here.
\begin{prop}\label{prop:EPa-EQa}
Let $\Phi_s$ be a section of $\Ind_{P_a(\A)}^{G(W_a)(\A)}\chi |\det |_{\A_E}^s$, which is $K_{G(W_a)}$-finite,
and $f\in\sigma\otimes\chi^{-1}$. Define
\begin{equation}\label{eq:Phi_fs}
  \Phi_{f, s}(g_a):= \int_{G(X)(\A)}  \Phi_s(\iota( g_a,g))f (g) dg.
\end{equation}
Then
\begin{enumerate}
\item $\Phi_{ f,s}(g_a)$ is absolutely convergent for $\Re s > ( m +a)/2$ and
has meromorphic continuation to the whole $s$-plane;
\item it is a
$K_{G(X_a)}$-finite section in $\cA_{a} (s,\chi, \sigma)$ and
\item we have
  \begin{equation*}
    \int_{[G(X)]} E^{P_a}(\iota(g_a,g),s,\Phi) f(g) dg = E^{Q_a}(g_a,s,\Phi_f).
  \end{equation*}
\end{enumerate}
\end{prop}

Now we show that the sections of the form $\Phi_{ f,s}$ are enough to detect the poles of Eisenstein series.
\begin{lemma}
  Let
  \begin{equation*}
    d^S (s) = \prod_{0 \le j < \frac{m}{2}} L_F^S (2s+m-2j,\chi) \prod_{0<j\le \frac{m}{2}} L_F^S (2s+m-2j+1,\varepsilon_{E/F}\chi).
  \end{equation*}
Then
\begin{equation*}
  d^S (s+ \frac{1}{2} a) ( L^S (s+\frac{1}{2} (1+a),\sigma \times \chi))^{-1}\Phi_{ f,s}(g_a)
\end{equation*}
is a holomorphic section of $\cA_{a} (s,\chi, \sigma)$.
\end{lemma}
\begin{proof}
Let $v$ be  a place of $F$. Assume that $f$ and $\Phi$ are decomposable. In order to see the normalising factor for $\Phi_{f,s} (g_a)$ we need to compute
\begin{equation*}
  \int_{G(X)(F_v)}  \form{\Phi_{v, s}(\iota(g_a,g))\sigma_v\otimes\chi_v^{-1} (g)f_v}{f'_v} dg
\end{equation*}
for all $f'_v$ in the dual of $\sigma_v$, or equivalently
\begin{equation*}
  \int_{G(X)(F_v)}  \Phi_{v, s}(\iota_1(g^{-1} g_a))\form{\sigma_v\otimes\chi_v^{-1}  (g)f_v}{f''_v} dg
\end{equation*}
for all $f''_v$ in the dual of $\sigma_v\otimes\chi_v^{-1}$.  Let $S$ be a finite set of places of $F$ away from which  $G (X)_v$, $\sigma_v$ and $\chi_v$ are unramified and $\Phi_v$, $f_v$ and $f_v''$ are spherical.

  For $a=0$, the unramified computation for unitary group was done in \cite [Thm.~3.1] {MR1166512} which follows the line of \cite [Part A, Sec.~6] {MR892097}.

  Consider the case where $a>0$. Fix $v\not\in S$ and suppress it from notation. Since $\Phi$ is $K_{G (W_a)}$-finite, we only need to consider holomorphicity for $g_a\in P_a (F_v)$. Then since $N_{P_a} (F_v)$ acts trivially on $\Phi$ by left translation, we may as well consider $g_a\in M_{P_a} (F_v)$ which is isomorphic to $R_{E/F}\GL_a (F_v)\times G (X) (F_v)$. When $g_a \in R_{E/F}\GL_a (F_v)$, the left translation of $g_a$ on $\Phi$ gives the factor $\chi|\ |_{E_v}^{s}\rho_{P_a} (g_a)$ which has no impact on holomorphicity. Thus we may assume $g_a\in G (X) (F_v)$. Then the problem is reduced to the case where $a=0$. Note that the difference between $\rho_P$ and $\rho_{P_a}$ gives a shift in $s$ by $a/2$.  Thus we prove our statement.
\end{proof}
\begin{lemma}\label{lemma:approx}
  For any $s_1 \in \R$, if $a$ is large enough, then for $\Re (s)>s_1$,
  every section in $\cA_a (s, \chi,\sigma)$ can be approximated by sections of the form $\Phi_{f,s}$ in \eqref{eq:Phi_fs}.
\end{lemma}
\begin{proof}
By the previous lemma, when $a$ is large enough, $\Phi_{f,s} (g_a)$ itself is a  holomorphic section in $\cA_a (s, \chi,\sigma)$ for $\Re ( s) > s_1$. 
The problem of approximation is local. Let $v$ be a place of $F$.
Let $Q_a^0$ be the subgroup of $Q_a$ with Levi part isomorphic to $R_{E/F}\SL_a \times \{I_X\}$ and  unipotent part isomorphic to $N_a$.  Similarly define $P_a^0$ to be the subgroup of $P_a$ with Levi part isomorphic to $R_{E/F}\SL_{m+a}$ and  unipotent part isomorphic to $N_{P_a}$.

We have a surjection from the space $C^\infty_c (Q_a^0 (F_v)\lmod G (X_a) (F_v), V_\sigma)$ of  smooth compactly supported functions valued in the space $V_\sigma$ of $\sigma$ to  $\cA_a (s,\chi,\sigma)$ given by
\begin{align}\label{eq:sec-in-E-s-sigma-chi}
    \xi &\mapsto \left(g_a\mapsto
  \int_{E_v^\times \times G (X) (F_v)} \chi^{-1} (t) |t|_{E_v}^{ -s} \rho_{Q_a} (\hat {t})
  \sigma_v (g^{-1}) \xi (\hat {t}gg_a) d^\times t\; dg\right)  ,
\end{align}
where for $t\in E_v^\times$, $\hat {t}$ denotes the element $\diag (t,I_{a-1}) \in R_{E/F}\GL_a (F_v)$. It can be checked that $K_{ a,v}$-finite elements in the space
\begin{equation*}
 C^\infty_c (Q_a^0 (F_v)\lmod G (X_a) (F_v), V_{\sigma,v})
\end{equation*}
can be approximated by elements from the space
\begin{equation*}
  C^\infty_c (Q_a^0 (F_v)\lmod G (X_a) (F_v))_{\text { $K_{ a,v}$-finite}}  \otimes V_{\sigma,v}.
\end{equation*}
Assume that
 $\xi \in C^\infty_c (Q_a^0 (F_v)\lmod G (X_a) (F_v), V_{\sigma,v})_{\text { $K_{ a,v}$-finite}}$ is approximated by a finite sum
 $\sum_i \xi_i \nu_i $ for $\xi_i\in C^\infty_c (Q_a^0 (F_v)\lmod G (X_a) (F_v))_{\text { $K_{ a,v}$-finite}}$ and $\nu_i \in V_{\sigma,v}$.

We note that by an argument in the original doubling method, the image of
\begin{equation*}
  F_v^\times \times Q_a^0 (F_v)\lmod G (X_a) (F_v) \hookrightarrow F_v^\times \times P_a^0 (F_v) \lmod G (W_a) (F_v)
\end{equation*}
is open and dense in the right-hand side in the Zariski topology.
Thus an element in $C^\infty_c (Q_a^0 (F_v)\lmod G (X_a) (F_v))$ can be naturally extended to an element in $C^\infty_c ( P_a^0 (F_v) \lmod G (W_a) (F_v))$. Furthermore $K_a$-finite functions are extended to $K_{G ( W_a)}$-finite functions.
Then we extend $\xi_i$ to $\til {\xi}_i \in C^\infty_c ( P_a^0 (F_v) \lmod G (W_a) (F_v))_{\text { $K_{G ( W_a)}$-finite}}$ from which we can construct a $K_{G ( W_a)}$-finite section $\Phi_s^{(i)}$ of $\Ind_{P_a (F_v)}^{G (W_a) (F_v)} \chi_v|\det|_{E_v}^s$ by
\begin{equation*}
  \Phi_s^{(i)} ( g_a) =
  \int_{E_v^\times} \chi^{-1} (t) |t|_{E_v}^{ -s} \rho_{P_a} (\hat {t})
  \til {\xi_i} (\hat {t}g_a) d^\times t .
\end{equation*}
Thus we see that the right hand side of  \eqref{eq:sec-in-E-s-sigma-chi} can be approximated by
\begin{align*}
  &\sum_i \int_{  G (X) (F_v)} \Phi_s^{(i)} (\iota_1 ( gg_a))
  \sigma_v (g^{-1})\nu_i  dg\\
=&\sum_i \int_{  G (X) (F_v)} \Phi_s^{(i)} (\iota_1 ( g^{-1} g_a)) \sigma_v (g)\nu_i  dg\\
=&\sum_i \int_{  G (X) (F_v)} \Phi_s^{(i)} (\iota ( g_a,g)) \chi^{-1} (g)\sigma_v (g)\nu_i  dg.
\end{align*}
This is what we are looking for.
\end{proof}

\begin{lemma}\label{lemma:section-general-enough}
Assume that $a$ is large enough and that $s_0$ is the maximal element in $\in\cP_1 (\sigma,\chi)$. Then
there exists a section of $\cA_{a} (s,\chi, \sigma) $ of the form $\Phi_{f, s}$ for $\Phi_s\in\Ind_{P_a(\A)}^{G(W_a)(\A)}\chi |\det |^s$ and $f\in \sigma\otimes\chi^{-1}$ such that $E^{Q_a} (g,s,\Phi_{f,s})$ has pole at $s=s_0$.
\end{lemma}
\begin{proof}
  This follows from  Prop.~\ref{prop:pole1->polea} and Lemma~\ref{lemma:approx}.
\end{proof}

The poles of $E^{P_a} (g,s,\Phi)$ are determined in \cite{MR1692899}.
\begin{prop}\label{prop:Siegel-eis-pole}
  The poles in the right half space $\Re (s) \ge 0$ of $E^{P_a} (g,s,\Phi)$ are at most simple and are contained in the set
  \begin{equation*}
    \Xi_{m+a} (\chi) =\{\frac{1}{2} (m+a-\epsilon_\chi) - j  |  j\in \Z, 0 \le j < \frac{1}{2} (m+a-\epsilon_\chi)\}.
  \end{equation*}
\end{prop}

Assume that $\epsilon_\chi=0$. We may associate to $\chi$ a character $\nu_\chi$ which at each local place is a character of $E_v^1$ defined by
\begin{equation}\label{eq:nu_chi}
  \nu_{\chi, v} (x/\bar {x}) = \chi_v (x).
\end{equation}
\begin{prop} \label{prop:non-siegel-eis-pole}
Let $\sigma\in\CA_\cusp(G(X))$ and $s_0$ be the maximal element in $\cP_1  (\sigma,\chi)$. Then the following hold.
  \begin{enumerate}
  \item  We have
  $s_0 = \frac{1}{2} (m+1-\epsilon_\chi)-j$ for
  some integer $j$ such that $0 \le j < \frac{1}{2} (m+1-\epsilon_\chi)$.
\item The maximal element $s_0$ can take the value $\frac {1} {2} (m+1)$ only when $\epsilon_\chi =0$, and in this case,
$\sigma$ is isomorphic to the one dimensional representation $\nu_\chi \circ \det$. 
  \end{enumerate}
\end{prop}

\begin{proof}
  Assume that $s_0$ is  the maximal element in $\cP_1  (\sigma,\chi)$. Then for $a$ large enough, by Lemma~\ref{lemma:section-general-enough}, there exists a section of $ \Ind_{Q_a(\A)}^{G(X_a)(\A)}\chi |\ |^s\otimes \sigma$ of the form $\Phi_{f, s}$ with $f\in \sigma\otimes\chi^{-1}$ such that $E^{Q_a} (g,s,\Phi_{f,s})$ has a pole at $s=s_0+\frac{1}{2} (a-1)$.   From Prop.~\ref{prop:EPa-EQa}  and Prop.~\ref{prop:Siegel-eis-pole}, we see that $s_0+\frac{1}{2} (a-1)$ is in
$ \Xi_{m+a} (\chi)$. Thus $s_0$ is of the form $\frac{1}{2}
(m+1-\epsilon_\chi) - j$ for some integer $j$ such that $0 \le j <
\frac{1}{2} (m+1-\epsilon_\chi)$.

The space of residues of $E^{P_a} (g,s,\Phi)$ at $s=\frac{1}{2} (m+a)$ is the one-dimensional representation $\nu_\chi\circ \det$. Then from Prop.~\ref{prop:EPa-EQa} in order to produce a non-vanishing residue for $E^{Q_a} (g_a,s,\Phi_{f,s})$, $\sigma\otimes\chi^{-1}$ must be the one-dimensional representation $\nu_\chi^{-1} \circ \det$. Hence $\sigma$ must be the one-dimensional representation $\nu_\chi \circ \det$.
\end{proof}

\section{First Occurrence in Theta Correspondences}
\label{sec:fo-theta}
\subsection{Unitary dual reductive pairs}
\label{sec:theta-lift-unitary}
Let $X$ be a skew-Hermitian space and $Y$ a Hermitian space. Then $G
(X)$ and $G (Y)$ form a dual reductive pair. Fix a pair of characters
$(\chi_1,\chi_2)$ of $E^\times \lmod\A_E^\times$ such that $\chi_1|_{\A_F^\times} =
\varepsilon_{E/F}^{\dim Y}$ and $\chi_2|_{\A_F^\times} =
\varepsilon_{E/F}^{\dim X}$. Along with the additive character $\psi$,
$\chi_1$ (resp. $\chi_2$) determines a splitting of the metaplectic group $\Mp(R_{E/F} ( Y\otimes_EX) ) (\A)$ 
over $G (X) (\A)$ (resp. $G (Y)(\A)$). Then we can define the Weil representation
$\omega_{\psi,\chi_1,\chi_2}$ of $G (X) (\A) \times G (Y) (\A)$ on the
Schwartz space $\cS (V (\A))$ where
$V$ is a maximal isotropic subspace of
$R_{E/F}(Y\otimes_EX)$.  See \cite{MR1286835} for more details.

If we change $\chi_2$ to another eligible character $\chi_2'$,
then (c.f. \cite [eq.~(1.8)] {MR1327161} and \cite [eq.~(2.9)] {MR3071813})
\begin{equation}
  \label{eq:change-chi2}
  \omega_{\psi,\chi_1,\chi_2'} (g,h) = \nu_{\chi_2'\chi_2^{-1}} (\det h)\omega_{\psi,\chi_1,\chi_2} (g,h)
\end{equation}
for $(g,h) \in  G (X) (\A) \times G (Y) (\A)$. The character $\nu_{\chi_2'\chi_2^{-1}}$ is defined in \eqref{eq:nu_chi}. Similar result holds if we change $\chi_1$.

For $\xi\in \cS (V (\A))$ we form the theta series
\begin{equation*}
  \theta_{\psi,\chi_1,\chi_2,X,Y} (g,h,\xi) :=
\sum_{v\in V (F)} \omega_{\psi,\chi_1,\chi_2} (g,h)\xi (v)
\end{equation*}
for $g\in G (X) (\A)$ and $h\in G (Y) (\A)$.
For $\sigma\in\CA_\cusp(G(Y))$, and for $f\in\sigma$ and $\xi\in \cS (V (\A))$, we define
\begin{equation*}
  \theta_{\psi,\chi_1,\chi_2,Y}^X (g,f,\xi) := \int_{[G (Y)]} \theta_{\psi,\chi_1,\chi_2,X,Y} (g,h,\xi) f (h)dh.
\end{equation*}
Then the  theta lift $\theta_{\psi,\chi_1,\chi_2, Y}^{X} (\sigma)$  is defined to be the span of all such 
$\theta_{\psi,\chi_1,\chi_2,Y}^X (g,f,\xi)$'s. 
Similarly if $\sigma\in\CA_\cusp(G(X))$, $\theta_{\psi,\chi_1,\chi_2, X}^{Y} (\sigma)$ is defined to be the span of
\begin{equation*}
  \theta_{\psi,\chi_1,\chi_2,X}^Y (g,f,\xi) := \int_{[G (X)]} \theta_{\psi,\chi_1,\chi_2,X,Y} (g,h,\xi) f (g)dg
\end{equation*}
for all  $f\in\sigma$ and $\xi\in \cS (V (\A))$. We note that the choice of $V$ is not important.

In the case where $\dim Y =0$ the Weil representation is simply the one-dimensional representation in which $G (X) (\A)$ acts by $\nu_{\chi_1}\circ\det$ and the theta series is given by
\begin{equation*}
  \theta_{\psi,\chi_1,\chi_2,X,Y} (g,I,\xi) := \omega_{\psi,\chi_1,\chi_2} (g,I)\xi (0)
= \nu_{\chi_1} (\det g) \xi (0)
\end{equation*}
for $g\in G (X) (\A)$.

\subsection{First occurrence and lowest occurrence}
\label{sec:folo}

Let $Y$ be a (possibly trivial) Hermitian  space over $E$ . Let $r$ be its Witt index. We define
similarly the Hermitian space $Y_b$ for $b\ge -r$ by adjoining or stripping away some copies of hyperbolic plane.
  Let $G (Y_b)$
be the isometry groups. Then $G (X)$ and $G (Y_b)$ form a dual
reductive pair and we consider the theta lift
$\theta_{\psi,\chi,\chi_2,X}^{Y_b} (\sigma)$ of $\sigma$ from $G
(X)$ to the Witt tower of unitary groups $G (Y_b)$ for varying
$b$'s. Note that $\chi$ is used for the splitting over $G (X)$ and
$\chi_2$ is used for the splitting over $G (Y_b)$ for all $b$.

Assume that for $b=b_0$, $\theta_{\psi,\chi,\chi_2,X}^{Y_b}
(\sigma)$ is nonzero and cuspidal. Then we say that the first
occurrence $\FO_{\psi,\chi,\chi_2}^Y (\sigma)$ of $\sigma$ in the
Witt tower of $G (Y)$ is $\dim (Y_{b_0})$. Once we fix $\chi$ the
parity of the dimension of $Y$ is fixed: $\epsilon_{\chi} \equiv \dim
Y \pmod {2}$. Note that $\FO_{\psi,\chi,\chi_2}^Y (\sigma) = \FO_{\psi,\chi,\chi_2'}^Y (\sigma)$ for any eligible $\chi_2'$, 
because the corresponding theta  lifts are off just by a twist of character. Thus we define the lowest occurrence
\begin{align*}
  \LO_{\psi,\chi} (\sigma) := \min  \{\FO_{\psi,\chi,\chi_2}^Y (\sigma) | Y: \dim Y \equiv  \epsilon_{\chi} \pmod 2\}.
\end{align*}
With this setup we can state the following.
\begin{thm}\label{thm:Eis-pole-LO}
For $\sigma\in\CA_\cusp(G(X))$, if $s_0$ is the maximal element of $\cP_1(\sigma,\chi)$, then
  \begin{enumerate}
  \item $s_0 = \frac{1}{2} (m+1-\epsilon_\chi)-j$ for
  some integer $j$ such that $0 \le j < \frac{1}{2} (m+1-\epsilon_\chi)$ ;
  \item $\LO_{\psi,\chi} (\sigma\otimes\chi^{-1}) \le  2j+\epsilon_\chi$;
    \item $2j+\epsilon_\chi \ge r_X$ where $r_X$ is the Witt index of $X$.
  \end{enumerate}
\end{thm}

\begin{rmk}
  The first part is just Prop.~\ref{prop:non-siegel-eis-pole}.
   The proof will be given  in Sec.~\ref{sec:pf-thm} by using  the regularised Siegel-Weil formula.
\end{rmk}

The following theorem is a corollary to Thm.~\ref{thm:Eis-pole-LO}.
\begin{thm}\label{thm:poleL}
For $\sigma\in\CA_\cusp(G(X))$, the following hold.
  \begin{enumerate}
  \item Assume that the partial $L$-function $L^S (s,\sigma \times \chi)$  has a pole at $s= \half (m+1-\epsilon_\chi) - j >0$  or assume that   $m-\epsilon_\chi$ is even and that $L^S (s,\sigma \times \chi)$ is non-vanishing at $s=\half (m+1-\epsilon_\chi) - j=1/2$. Then $\LO_{\psi,\chi} (\sigma\otimes\chi^{-1}) \le 2j+ \epsilon_\chi$.
  \item If $\LO_{\psi,\chi} (\sigma\otimes\chi^{-1}) = 2j + \epsilon_\chi < m +1$, then $L^S (s,\sigma \times \chi)$ is holomorphic for $\Re (s ) > \half (m+1-\epsilon_\chi) - j$.
  \item If $\LO_{\psi,\chi} (\sigma\otimes\chi^{-1}) = 2j + \epsilon_\chi \ge m +1$, then $L^S (s,\sigma \times \chi)$ is holomorphic for $\Re (s ) \ge 1/2$.
  \end{enumerate}
\end{thm}
\begin{proof}
  Assume that $L^S (s,\sigma \times \chi)$  has a pole at $s= \half (m+1-\epsilon_\chi) - j >0$. Let $\half (m+1-\epsilon_\chi) - j'$ be its maximal pole. Then by Prop.~\ref{prop:L-pole->Eis-pole}, $\half (m+1-\epsilon_\chi) - j' \in \cP_1 (\sigma,\chi)$, i.e., the Eisenstein series $E^{Q_1} (g,s,f)$ has a pole there. Let $\half (m+1-\epsilon_\chi) - j''$ be the maximal member of $\cP_1 (\sigma,\chi)$. We note that $j'' \le j' \le j$. By Thm.~\ref{thm:Eis-pole-LO}, the lowest occurrence $\LO_{\psi,\chi} (\sigma) \le 2j'' + \epsilon_\chi$. Hence $\LO_{\psi,\chi} (\sigma) \le 2j + \epsilon_\chi$.

Assume that  $L^S (s,\sigma \times \chi)$ is non-vanishing at $s=1/2$ and that   $m-\epsilon_\chi$ is even. If $L^S (s,\sigma \times \chi)$ has a pole in $\Re s > 1/2$, this reduces to the case in the preceding paragraph. Thus we assume that $L^S (s,\sigma \times \chi)$ is holomorphic in  $\Re s > 1/2$. In this case $j=\half (m-\epsilon_\chi)$. Again by Prop.~\ref{prop:L-pole->Eis-pole}, we have $1/2=\half (m+1-\epsilon_\chi) - j \in \cP_1 (\sigma,\chi)$. Then the same reasoning as in the preceding paragraph shows that $\LO_{\psi,\chi} (\sigma) \le 2j + \epsilon_\chi = m$. This concludes the proof of part (1).

For part (2),  assume that $\LO_{\psi,\chi} (\sigma\otimes\chi^{-1}) = 2j + \epsilon_\chi < m +1$ and  that $L^S (s,\sigma \times \chi)$ has a pole in $\Re (s ) > \half (m+1-\epsilon_\chi) - j$. Part (1) shows that $\LO_{\psi,\chi} (\sigma\otimes\chi^{-1}) \le 2j'+ \epsilon_\chi$ for some $j' < j$. Thus $\LO_{\psi,\chi} (\sigma\otimes\chi^{-1})$ is strictly less than $2j+ \epsilon_\chi$, which is a contradiction.

For part (3), assume that $\LO_{\psi,\chi} (\sigma\otimes\chi^{-1}) = 2j + \epsilon_\chi \ge m +1$ and that $L^S (s,\sigma \times \chi)$ has a pole in $\Re (s )\ge 1/2$. Then part (1) implies that $\LO_{\psi,\chi} (\sigma\otimes\chi^{-1})\le m$. We get a contradiction.
\end{proof}

\subsection{Degenerate principal series representation}
\label{sec:deg-prin-rep}

We summarise the results on the structure of degenerate principal series representation here to prepare for the proof of Thm.~\ref{thm:Eis-pole-LO}. We work in the local case in this subsection.
Fix a place $v$ of $F$. First assume that $v$ is non-split. Let $W$ be a split skew-Hermitian space of dimension $2n$ over $E_v$, so $G (W) $ is a quasi-split unitary group.
Let $P$ be a Siegel parabolic subgroup of $G (W)$.
The irreducible components of the induced representation $I (s,\chi):=\Ind_P^{G (W)}\chi|\ |_E^s$ are related to theta correspondence as we explain below. Let $Y$ be a Hermitian space over $E_v$ of dimension $l$ such that $l \equiv \epsilon_\chi \mod 2$.
We consider the dual reductive pair of $G (Y)$ and the quasi-split unitary group $G (W)$.
As we reviewed, the Weil representation for unitary dual reductive pairs depends on the choice of two characters $\chi_1$ and $\chi_2$ in addition to the additive character $\psi$.  We emphasise our convention  that $\chi_1$  is used for the  splitting over $G (W)$ and $\chi_2$ is used for the  splitting over  $G (Y)$. Since $\dim W$ is even, we may and do let $\chi_2 = \mathbf {1}$. Also we set $\chi_1=\chi$. Then we  get the corresponding local Weil representation of the dual reductive pair.

Next assume that $v$ is split. Then $E_v=F_v\oplus F_v$. The non-trivial Galois element of $E/F$ exchanges the two factors of $E_v$. In fact the above discussion for $E_v/F_v$ with non-split $v$ actually goes through. The skew-Hermitian space $W$ can be viewed as $W_1 \oplus W_2$ where $W_1$ and $W_2$ are vector spaces over $F_v$. The non-degenerate skew-Hermitian form on $W$ induces a non-degenerate pairing of $W_1$ and $W_2$. Thus $G (W)$ consists of elements in $\GL (W_1) \times \GL (W_2)$ that are of the form $(g, ( g^*)^{-1})$ where $g^*$ is the adjoint of $g$ under the pairing of $W_1$ and $W_2$. Projecting $G (W)$ to the factor $\GL (W_1)$ we get an isomorphism of $G (W)$ with $\GL (W_1)$. Similarly writing $Y$ as $Y_1\oplus Y_2$ for  two vector spaces $Y_1$ and $Y_2$ over $F$ we can identify $G (Y)$ with  $\GL (Y_1)$. The Siegel parabolic subgroup $P$ of $G (W)$ under this isomorphism becomes the parabolic subgroup $P_{n,n}$ of $\GL_{2n}$ corresponding to the partition $[n,n]$ of $2n$. Write $\chi$ as $\chi^{(1)}  \times \chi^{(2)}$. Then $\Ind_P^{G (W)}\chi|\ |_E^s$ is identified with $\Ind_{P_{n,n}}^{\GL (W_1)} \chi^{(1)}|\ |_F^s\otimes (\chi^{(2)})^{-1} |\ |_F^{-s}$.

When $v$ is non-split let $R_{\psi,\chi} (Y)$ be the image of the map of  forming ``Siegel-Weil section'':
\begin{align*}
  \cS (R_{E_v/F_v} ( Y\otimes_E W^+)) &\rightarrow I (s,\chi)\\
  \xi &\mapsto \omega_{\psi,\chi,\mathbf {1},W,Y} (g) \xi (0)
\end{align*}
where $W^+$ is a maximal isotropic subspace of $W$.
Up to isometry there are two Hermitian spaces with given dimension when $v$ is  non-archimedean.

When $v$ is split let $R_{\psi,\chi} (Y_1)$ be the image of the map of  forming ``Siegel-Weil section'':
\begin{align*}
  \cS ( Y_1\otimes_F W_1) &\rightarrow I (s,\chi)\\
  \xi &\mapsto \omega_{\psi,\chi,\mathbf {1},W,Y} (g) \xi (0).
\end{align*}

From \cite{MR1459856} for the non-archimedean case and \cite{MR1443883,MR2378075} for the archimedean case we have:
\begin{prop} Let
    $0\le s=\frac{1}{2} (n-\epsilon_\chi)-j < \frac{1}{2} n$.
  \begin{enumerate}
  \item Assume that $v$ is non-archimedean and $E_v/F_v$ is not split.  Then the maximal semi-simple
    submodule of $I (-s,\chi)$ is
  \begin{align*}
    \oplus_{Y} R_{\psi,\chi} (Y)
  \end{align*}
  where $Y$ runs over Hermitian spaces over $E_v$ of dimension
  $2j+\epsilon_\chi$.
\item Assume that $v$ is non-archimedean and $E_v/F_v$ is split. Then the maximal semi-simple
    submodule of $I (-s,\chi)$ is
  \begin{align*}
    R_{\psi,\chi} (Y)
  \end{align*}
  where $Y$ is the vector  spaces over $F_v$ of dimension
  $2j+\epsilon_\chi$.
  \item Assume that $E_v/F_v=\C/\R$.  Then the maximal semi-simple
    submodule of $I (-s,\chi)$ is
  \begin{align*}
    \oplus_{Y} R_{\psi,\chi} (Y)
  \end{align*}
where $Y$ runs over Hermitian spaces over $\C$ of dimension $2j + \epsilon_\chi$.
\item Assume that $F_v=\C$, so $E_v/F_v$ split. Then the maximal semi-simple
    submodule of $I (-s,\chi)$ is
 \begin{align*}
     R_{\psi,\chi} (Y)
  \end{align*}
where $Y$ is a vector space over $\C$ of dimension $2j + \epsilon_\chi$.
\item The spaces $R_{\psi,\chi} (Y)$ above are irreducible and unitarisable.
  \end{enumerate}
\end{prop}
\begin{rmk}
  The norm $|\ |_\C$ that we use is the square of the usual norm of $\C$, whereas the  usual norm of $\C$ is used in \cite{MR2378075}.
\end{rmk}

\section{Proof of Theorem~\ref{thm:Eis-pole-LO}}
\label{sec:pf-thm}
We return to the global case. Temporarily, let $G$ denote the quasi-split unitary group $\mathrm {U} (n,n)$. Let $B$ be its Borel subgroup and  $P$ the standard Siegel parabolic subgroup with Levi subgroup $M$ and unipotent radical $N$.

As defined in the last section, $R_{\psi,\chi} (Y_v)$ is regarded as a subquotient of $\Ind_{P (F_v)}^{G (F_v)} \chi |\ |_{E_v}^s$
for certain values of $s$, where $Y_v$  is a Hermitian space over $E_v$ with $\dim Y_v \le n$ and $\dim Y_v \equiv \epsilon_\chi \pmod 2$.

 Let $\cC = \{Y_v\}_v$ be a collection of such Hermitian spaces $Y_v$ of the same dimension and  let $\Pi_{\psi,\chi} (\cC)$ denote the representation
\begin{equation*}
  \otimes_v R_{\psi,\chi} (Y_v).
\end{equation*}
 When the collection $\cC$ comes from a global Hermitian space $Y$ over $E$, then we will denote $\Pi_{\psi,\chi} (\cC)$ by $\Pi_{\psi,\chi} (Y)$. When $\cC$ does not come from a global Hermitian space then we say that $\cC$ is incoherent. We  note that $\cC$ is coherent if and only if $\prod_v \epsilon_{E_v/F_v} (\disc (Y_v))=1$.

 \begin{prop}\label{prop:no-incoherent}
  Assume that  $\cC$ is an incoherent collection of Hermitian spaces of dimension $\ell$ with $0< \ell \le n$ and $\ell \equiv \epsilon_\chi \pmod 2$. Then
  \begin{equation*}
    \dim \Hom_{G} (\Pi_{\psi,\chi} (\cC),\cA (G)) =0.
  \end{equation*}
\end{prop}
\begin{proof}
 Given  an $n\times n $-Hermitian matrix $\beta$ with entries in $E_v$, let $\psi_\beta$  be the character of $N$ given by
  \begin{equation*}
    \psi_\beta (n (b)) = \psi (\frac{1}{2}\tr_{E/F} (\tr (\beta b)) ).
  \end{equation*}
Let $W_\beta$ be the Whittaker functional on $\cA (G)$ defined as follows:
\begin{equation*}
  W_\beta (f) = \int_{[N]} f (n)\psi_\beta^{-1} (n)dn.
\end{equation*}
Assume that $D \in \Hom_{G} (\Pi_{\psi,\chi} (\cC),\cA (G))$. In order for $D$ to be nonzero, by an analogous argument of the proof of \cite [Lemma~2.2] {MR1150600} or \cite [Thm.~3.1] {kudla94:_regul_siegel_weil_formul},  there must exist a $\beta \in \Her_n (F)$ with $\rank \beta = \ell$  such that $W_\beta\circ D \neq 0$. Let $v$ be a place of $F$. Then by \cite [Lemma~4.4] {MR1459856}, $W_\beta\circ D_v$ can be nonzero if and only if $\beta$ is represented by $Y_v$.  Thus we need that $\disc (Y_v)$ is equal to the pseudo-determinant $\pdet (\beta)$ of $\beta$ modulo $\bN E_v^\times$. Then
 \begin{equation*}
   \prod_v \epsilon_{E_v/F_v} (\disc (Y_v)) =    \prod_v \epsilon_{E_v/F_v} (\pdet\beta) =1
 \end{equation*}
by the product formula. This contradicts the fact that $\cC$ is incoherent. Thus $D$ must be the zero map.
\end{proof}

\begin{prop} \label{prop:residue-square-integrable}
  Let $s_0 = \frac{1}{2} (n-\ell) \in \Xi_{n} (\chi)$. For  a $K$-finite section $\Phi_s$ of $\Ind_{P (\A)}^{G (\A)}\chi|\ |_{\A_E}^s$,
the residue $ \res_{s=s_0} E^{P} (g,\Phi_s) $ is a square-integrable automorphic form.
\end{prop}
\begin{proof}
When $\ell=0$ the space of residues is isomorphic to the one-dimensional representation $\nu_\chi\circ \det$. Thus we assume $\ell>0$.

First we note that $ \res_{s=s_0} E^{P} ( g,\Phi_s) $ is concentrated on the Borel subgroup $B$ of $G $. Thus by a criterion of Langlands we only need to check if the exponents of the constant term of $ \res_{s=s_0} E^{P} ( g,\Phi_s) $ along $B$ lie in the negative obtuse Weyl chamber  of $G$.
The constant term $E^{P}_B ( g,\Phi_s)$ is given by
\begin{equation*}
    \sum_{w_T} M (w_T,s)  \Phi_s (g) .
\end{equation*}
via an analogue of \eqref{eq:const-term-eis}. The sum is over all subsets $T$ of $\{1,\ldots,n\}$. We will also adopt the notation introduced there. In particular we note that $T$ consists of elements $\{i_0<i_1<\cdots < i_t\}$ and $w_T$ is given by \eqref{eq:w_T1} and \eqref{eq:w_T2} where we change $a$ to $n$.
The section $\Phi_s$ lies in
\begin{equation*}
  \Ind_{B (\A)}^{G (\A)} \chi|\ |_{\A_E}^{s-\frac{1}{2} (n-1)}\otimes \chi|\ |_{\A_E}^{s-\frac{1}{2} (n-3)} \otimes \cdots \otimes \chi|\ |_{\A_E}^{s-\frac{1}{2} (1-n)} .
\end{equation*}
Assume that $M (w_T,s_0)  \Phi_{s_0} (g)$ lies in
\begin{equation*}
  \Ind_{B (\A)}^{G (\A)} \chi|\ |_{\A_E}^{x_1}\otimes \chi|\ |_{\A_E}^{x_2} \otimes \cdots \otimes \chi|\ |_{\A_E}^{x_n} .
\end{equation*}
It is clear that the exponents of $M (w_T,s)  \Phi_s (g)$ lie in the negative obtuse Weyl chamber if and only if $\sum_{k\le b} x_k < 0$ for all $1\le b \le n$.

Let $- e = s_0-\frac{1}{2} (n-1) = \frac{1}{2} (1-\ell)$ and $e'= s_0-\frac{1}{2} (1-n) = \frac{1}{2} (2n-\ell-1)$. Let $c$ be the maximal integer such that $i_c \le b$.  Then
\begin{equation*}
  \sum_{k\le b} x_k  = \sum_{k=1}^c (-e +k -1) - \sum_{k'=1}^{b-c} (e'-k'+1).
\end{equation*}
The second sum, if nonempty, is positive:
\begin{equation*}
   \sum_{k'=1}^{b-c} (e'-k'+1) = \frac {1} {2} (b-c) (2n-\ell- (b-c)) \ge \frac {1} {2} (b-c) (n-\ell) >0 .
\end{equation*}
If the first sum is positive, then $-e + (-e +c -1) >0$, which implies that $c> 2e +1$. In this case the unramified computation gives the normalising factor:
\begin{equation}\label{eq:L-w_T}
  \prod_{t<j\le n} \frac{L_F^S (-2e + 2j -2,\chi)}{L_F^S (-2e + 2j -2+1,\chi)}\\
  \cdot\prod_{t < j \le n}\prod_{j_{n-j+1} \le i \le t} \frac{L_E^S (-2e+i+j-2,\chi)}{L_E^S (-2e+i+j-2+1,\chi)}.
\end{equation}
We note that
\begin{align*}
  -2e + 2j -2 \ge -2e + 2t \ge -2e + 2c > -2e + 4e +2 >1,\\
  -2e+i+j-2 \ge -2e +  t \ge -2e + c > -2e + 2e +1 =1.
\end{align*}
Hence $M (w_T,s)  \Phi_{s} (g)$ is in fact holomorphic at $s=s_0$.

Now assume that the second sum is empty. Thus $c=b$. If the first sum is non-negative, from
\begin{equation*}
  \sum_{k=1}^c (-e +k -1) = \sum_{k=1}^b (-e +k -1) = \frac{1}{2}b (-2e +b -1)\ge 0,
\end{equation*}
we get $b\ge 2e+1$. We check that
\begin{align*}
  -2e + 2j -2 &\ge -2e + 2t \ge -2e + 2b \ge -2e + 4e +2 >1;\\
  -2e+i+j-2 &\ge -2e + b+1 +  t \\
            &\ge -2e + 2b +1 \ge -2e + 4e +2 +1 >1.
\end{align*}
Still using \eqref{eq:L-w_T}, we find that $M (w_T,s)  \Phi_{s} (g)$ is holomorphic at $s=s_0$. Thus we get a nonzero residue only when the exponents are in the negative obtuse Weyl chamber.
\end{proof}

Taking residue of the Eisenstein series at the point $s=\frac{1}{2} (n-\ell)$ in $\Xi_{n} (\chi)$ yields a $G (\A)$-intertwining map
\begin{equation*}
  A_{-1}: \Ind_{P (\A)}^{G (\A)}\chi|\ |_{\A_E}^s \rightarrow \cA (G).
\end{equation*}

\begin{cor}
  When $0\le \ell < n$, the $G (\A)$-intertwining map $A_{-1}$ factors through
\begin{equation*}
  \oplus_{Y} \Pi_{\psi,\chi} (Y)
\end{equation*}
where $Y$ runs over global Hermitian spaces of dimension $\ell$.
\end{cor}
\begin{proof}
By Prop.~\ref{prop:residue-square-integrable}, $A_{-1}$ factors through $\oplus_{\cC= \{Y_v\}_v} \Pi_{\psi,\chi} (Y_v)$. By Prop.\ref{prop:no-incoherent} we know that $A_{-1}$ further factors through $\oplus_{Y} \Pi_{\psi,\chi} (Y_v)$ where $Y$ is a global Hermitian space of dimension $\ell$.
\end{proof}

Now we go back to our usual notation. Recall that $s_0=\frac{1}{2} (m+1-\epsilon_\chi)-j$ is the maximal member in $ \cP_1 (\sigma,\chi)$. Then the above discussion and  the regularised Siegel-Weil formula in \cite{MR2064052} shows that if $\dim Y < m+a$, we have
\begin{multline}\label{eq:siegel-weil}
  \res_{s=s_0+\frac{1}{2} (a-1)} E^{P_a} (\iota (g_a,g),\Phi_s)\\
 =\sum_{Y:\dim Y = 2j+\epsilon_\chi} c_Y \int_{[G (Y)]} \theta_{\psi,\chi,\mathbf {1}} (\iota (g_a,g),h,\omega_{\psi,\chi,\mathbf {1}} (\alpha_Y)\phi_Y)dh.
\end{multline}
Here $\alpha_Y$ is an element in the local Hecke algebra of $G (W_a)$ at one local place
that is used to regularise the theta integral; $c_Y$ is some non-zero constant; $\phi_Y$ is some $K_{G (W_a)}$-finite Schwartz function in $\cS (Y^{m+a} (\A_E))$.

Finally we are ready to complete the proof of Thm.~\ref{thm:Eis-pole-LO}.
\begin{proof}[Proof of Thm.~\ref{thm:Eis-pole-LO}]
  Assume that $s=s_0=\frac{1}{2} (m+1-\epsilon_\chi)-j$ is the maximal member in $\cP_1 (\sigma,\chi)$. By Prop.~\ref{prop:pole1->polea}, $s=\frac{1}{2} (m+a-\epsilon_\chi)-j$ is a member in $\cP_a (\sigma,\chi)$. Thus by Lemma~\ref{lemma:section-general-enough}, there exists a section of $\cA_a (s,\chi,\sigma)$ of the form $\Phi_{f,s}$ (c.f. \eqref{eq:Phi_fs}) such that $E^{Q_a} (g,s,\Phi_{f,s})$ has a pole at $s=s_0$. Thus by Prop.~\ref{prop:EPa-EQa}, the Siegel Eisenstein series $E^{P_a}(\iota(g_a,g),s,\Phi)$ must have a pole at $s=\frac{1}{2} (m+a-\epsilon_\chi)-j$ for the section $\Phi$ of $\Ind_{P_a(\A)}^{G(W_a)(\A)}\chi |\det |_{\A_E}^s$. Then by \eqref{eq:siegel-weil}
 for $a$ big enough, there exists a Hermitian space $Y$ with
$\dim Y = 2j +\epsilon_\chi$ such that
\begin{equation}\label{eq:theta-integral}
   \int_{[G (Y)]} \theta_{\psi,\chi,\mathbf {1}} (\iota (g_a,g),h,\omega_{\psi,\chi,\mathbf {1}} (\alpha_Y)\phi_Y)dh
\end{equation}
is non-vanishing for some $\phi_Y \in \cS (Y^{m+a} (\A_E))$.
We may just assume that $\omega_{\psi,\chi,\mathbf {1}} (\alpha_Y)\phi_Y = \phi_Y^{(1)}\otimes \phi_Y^{(2)}$ for $\phi_Y^{(1)} \in \cS ((R_{E/F} (Y\otimes_E X_a))^+ (\A) )$ and $\phi_Y^{(2)} \in \cS ((R_{E/F} (Y\otimes_E X))^+ (\A) )$ where $^+$ denotes taking a maximal isotropic subspace. Let $\chi_2$ be a character of $E^\times\lmod \A_E^\times$ such that $\chi_2|_{\A_F^\times} = \epsilon_{E/F}^{m}$.
We separate variables in \eqref{eq:theta-integral} to get (c.f. \cite [Eq.~(5.3), (5.4)] {MR3071813})
\begin{equation*}
  \int_{[ G (Y)]} \chi_2^{-1} (\det h) \theta_{\psi,\chi,\chi_2} (g_a,h,\phi_Y^{(1)})
    \theta_{\psi,\chi,\chi_2} (g,h,\phi_Y^{(2)}) dh  .
\end{equation*}
Thus the residue $\res_{s=s_0+\frac{1}{2} (a-1)} E^{Q_a} ( g_a,\Phi_{f,s})$ is equal to
\begin{align*}
  &\int_{[G (X)]} \int_{[ G (Y)]} \chi_2^{-1} (\det h) \theta_{\psi,\chi,\chi_2} (g_a,h,\phi_Y^{(1)})
    \theta_{\psi,\chi,\chi_2} (g,h,\phi_Y^{(2)}) f (g) dh  dg\\
=& \int_{[ G (Y)]} \chi_2^{-1} (\det h) \theta_{\psi,\chi,\chi_2} (g_a,h,\phi_Y^{(1)}) \int_{[G (X)]}
    \theta_{\psi,\chi,\chi_2} (g,h,\phi_Y^{(2)}) f (g) dg dh . \\
\end{align*}
The non-vanishing of the residue implies that the inner integral is non-vanishing.
 Since it is exactly the theta lift of $f\in \sigma\otimes\chi^{-1}$ from $G (X)$ to $G (Y)$, we obtain that
 $\LO_{\psi,\chi,\chi_2} (\sigma\otimes\chi^{-1}) \le 2j+\epsilon_\chi$ for any $\chi_2$ with $\epsilon_{\chi_2}\equiv m \pmod 2$. Thus we conclude the second part of the theorem.

For the third part, assume that the lowest occurrence is realised in the Witt tower of $Y$ so that $\FO_{\psi,\chi,\chi_2}^Y (\sigma\otimes\chi^{-1}) = 2j'+ \epsilon_\chi$ for $j'\le j$. Let $\pi = \theta_{\psi,\chi,\chi_2,X}^Z (\sigma\otimes\chi^{-1})$ for $Z$ in the Witt tower of $Y$ with dimension $2j'+ \epsilon_\chi$. Then $\pi$ is cuspidal and hence by \cite [Thm.~5.3] {MR3071813} is irreducible.   It is known that $\FO_{\psi^{-1},\chi,\chi_2,Z}^{X} (\chi_2^{-1}  \otimes \pi) \le m_0 + 2( 2j'+ \epsilon_\chi)$ where $m_0$ is the dimension of the anisotropic kernel of $X$. On the other hand by \cite [Thm.~5.2] {MR3071813}
\begin{equation*}
  \theta_{\psi^{-1},\chi,\chi_2,Z}^{X} (\chi_2^{-1}  \otimes \pi) =
\theta_{\psi^{-1},\chi,\chi_2,Z}^{X} (\chi_2^{-1} \otimes \theta_{\psi,\chi,\chi_2,X}^Z (\sigma\otimes\chi^{-1})) = \sigma.
\end{equation*}
Thus $m\le m_0 + 2( 2j'+ \epsilon_\chi)$ from which we conclude that $r_X \le 2j'+ \epsilon_\chi \le 2j + \epsilon_\chi $.
\end{proof}

\section{Certain Periods and Theta Correspondences}
\label{sec:period}
\subsection{Periods and the first occurrence}
Let $\sigma\in\CA_\cusp(G(X))$ and $Y_0$ a (possibly trivial) anisotropic Hermitian space. We will consider the first occurrence of $\sigma$ rather than $\sigma\otimes\chi^{-1}$ in this section. The first occurrence $\FO_{\psi,\chi_1,\chi_2,X}^{Y_0} (\sigma)$ puts constraints on periods of the product of a cuspidal automorphic form in $\sigma$ and a theta series on $X$ and $Y_0$.  In some sense, $\sigma$ is $\theta_{\psi,\chi_1,\chi_2,X,Y_0} $-distinguished. To alleviate notation, we will generally suppress the dependence of the Weil representation on $\psi$, $\chi_1$ and $\chi_2$ and  taking $F$-points of the various unitary groups in this section.

\begin{prop}\label{prop:period-X->Y}
  Let $\sigma\in\CA_\cusp(G(X))$ and $Y_0$ a (possibly trivial) anisotropic Hermitian space. Assume that
  \begin{equation*}
    \FO_{\psi,\chi_1,\chi_2,X}^{Y_0} (\sigma) = \dim Y_0 + 2r
  \end{equation*}
 with $r \le \dim X$ and that the first occurrence is realised by $Y$.  Then
  the following hold.
  \begin{enumerate}
  \item There exist a non-degenerate subspace $Z$  of $X$ with dimension $\dim X - r$, a cusp form $f\in \sigma$ and a Bruhat-Schwartz function
  $\phi \in \cS (R_{E/F} (Y_0\otimes X)^+ (\A))$ such that
  \begin{equation}
    \label{eq:key-period}
    \int_{[G (Z)]} f (g) \theta_{\psi,\chi_1,\chi_2,X,Y_0} (g,1,\phi) dg
  \end{equation}
is non-vanishing.
\item For any subspace $Z$ of $X$ with dimension greater than $\dim X - r$, the analogous period integral always vanishes.
\item The period integrals of the form \eqref{eq:key-period} converge absolutely.
  \end{enumerate}
\end{prop}
\begin{rmk}
  If $Y_0= \{0\}$ then the theta series reduces to a constant times $\nu_{\chi_1}\circ \det (g)$. Compare with Prop.~5.2 in \cite{MR2540878}. When $r=0$ the period integral  is just the integral for theta lift of $\sigma$ to $G (Y_0)$.
In fact the above integrals can be replaced by
\begin{equation*}
  \int_{[G (Z)]} f (g) \theta_{\psi,\chi_1,\chi_2,Z,Y_0} (g,1,\phi) dg.
\end{equation*}
\end{rmk}
The symmetric version where we consider theta lift from $G (Y)$ to $G (X)$ is as follows.
\begin{prop}
\label{prop:period-Y->X}
Let $\sigma\in\CA_\cusp(G(Y))$ and $X_0$ a (possibly trivial) anisotropic skew-Hermitian space. Assume that
\begin{equation*}
\FO_{\psi,\chi_1,\chi_2,Y}^{X_0} (\sigma)  = \dim X_0  + 2r
\end{equation*}
 with $r \le \dim Y$ and that the first occurrence is realised by $X$. Then the following hold.
  \begin{enumerate}
  \item There exist a non-degenerate subspace $Z$  of $Y$ with dimension $\dim Y - r$,  a cusp form $f\in \sigma$ and a Bruhat-Schwartz function $\phi \in \cS (R_{E/F} (Y \otimes X_0)^+(\A))$ such that
  \begin{equation}
    \label{eq:key-period-dual}
    \int_{[G (Z)]} f (h)\theta_{\psi,\chi_1,\chi_2,X_0,Y} (1,h,\phi)dh,
  \end{equation}
is non-vanishing.
\item For any subspace $Z$ of $Y$ with dimension greater than $\dim Y - r$, the analogous period integral always vanishes.
\item The period integrals of the form \eqref{eq:key-period-dual} converge absolutely.
  \end{enumerate}
\end{prop}

\begin{proof}
To make use of Prop.~3.1 of \cite{MR3071813} with minimal change of notation, we prove the statement in Prop.~\ref{prop:period-Y->X}.

Since $f$ is cuspidal, it is rapidly decreasing over the Siegel domain of $G (Y)$. Thus the whole integrand is rapidly decreasing over the Siegel domain of $G (Y)$. It can be checked that  Siegel domain of $G (Z)$ is contained in some Siegel domain of $G (Y)$. Thus the whole integrand is rapidly decreasing over the Siegel domain of $G (Z)$. We conclude that the period integrals are absolutely convergent.

Next we relate the period integral to Fourier coefficients of theta lift.
We write $X$ as $\ell_r^+\oplus X_0 \oplus \ell_r^-$. Choose dual bases $e_1^+,\ldots, e_r^+$  for $\ell_r^+$ and $e_1^-,\ldots, e_r^-$ for $\ell_r^-$. Let $Q$ be the parabolic subgroup of $G (X)$ that stabilises $\ell_r^-$ and $N$ the unipotent radical of $Q$. Let $N'$ be the subgroup of $N$ that acts as identity on $X_0$. In fact $N'$ can be identified with $(r\times r)$-Hermitian matrices $\Her_r$.

Fix a polarisation of $R_{E/F} (Y\otimes X_0)$:
\begin{equation*}
  R_{E/F} (Y\otimes X_0) = ( R_{E/F} (Y\otimes X_0))^+ \oplus ( R_{E/F} (Y\otimes X_0))^-.
\end{equation*}
Then the  polarisation of $R_{E/F} (Y\otimes X)$ is chosen to be
\begin{align*}
  R_{E/F} (Y\otimes X)= (R_{E/F} (Y\otimes X))^+ \oplus (R_{E/F} (Y\otimes X))^-
  \end{align*}
with
\begin{align*}
  (R_{E/F} (Y\otimes X))^+ &= R_{E/F} (Y\otimes \ell_r^+) \oplus ( R_{E/F} (Y\otimes X_0))^+;\\
  (R_{E/F} (Y\otimes X))^- &= R_{E/F} (Y\otimes \ell_r^-) \oplus ( R_{E/F} (Y\otimes X_0))^-.
\end{align*}

Let $\Phi \in \cS ((R_{E/F} (Y\otimes X))^+ (\A))$. Let $c$ be a non-degenerate $(r\times r)$-Hermitian matrix.  Define an additive character of $ \Her_r(\A)$ as follows
\begin{equation*}
  \psi_{E,c} (\beta) = \psi_E (\tr (\beta c)).
\end{equation*}
Consider the $c$-th Fourier coefficient of the theta lift of $f$. It is given by
\begin{equation}\label{eq:Fourier-coeff-of-theta-lift}
  \int_{[N']} \int_{[G (Y)]} \theta_{\psi,\chi_1,\chi_2,X,Y} (ng,h,\Phi)f (h) \overline {\psi_{E,c} (\beta)} dh dn.
\end{equation}
Here $n=n (\beta) \in N' (\A)$ for $\beta \in \Her_r (\A)$.
We plug in the definition of theta series in \eqref{eq:Fourier-coeff-of-theta-lift} to get
\begin{multline*}
  \int_{[\Her_r ]} \int_{[G (Y)]} \sum_{y\in R_{E/F} (Y\otimes \ell_r^+)}
  \sum_{w\in (R_{E/F} (Y\otimes X_0))^+}
  \omega_{X,Y} (n (\beta)g,h)\Phi (y,w) \\
\cdot f (h)  \overline {\psi_{E,c} (\beta)} dh d\beta .
\end{multline*}
Since
\begin{equation*}
  \omega_{X,Y} (n (\beta),1)\Phi (y,w) = \Phi (y,w) \psi_E (\half \tr \form{y}{y}_Y\beta),
\end{equation*}
integration over $\beta$ vanishes unless $\form{y}{y}_Y = 2c$. Fix $y^0 \in Y^r$ that represents $2c$. Let $Z$ be the orthogonal complement to $y^0$ in $Y$. Then the integral above is equal to
\begin{align*}
&   \int_{[G (Y)]} \sum_{\substack{y\in R_{E/F} (Y\otimes \ell_r^+)\\ \form{y}{y}_Y = 2c}} \sum_{w\in (R_{E/F} (Y\otimes X_0))^+}
  \omega_{X,Y} (g,h)\Phi (y,w) f (h)
   dh   \\
=&  \int_{[G (Y)]} \sum_{\gamma \in G (Z) \lmod G (Y) } \sum_{w\in (R_{E/F} (Y\otimes X_0))^+}
  \omega_{X,Y} (g,h)\Phi (\gamma^{-1} y^0,w) f (h)
   dh .
\end{align*}
 Since we take summation over $w$, the above is equal to
\begin{align*}
  &  \int_{[G (Y)]} \sum_{\gamma \in G (Z)\lmod G (Y)} \sum_{w\in (R_{E/F} (Y\otimes X_0))^+}
  \omega_{X,Y} (g,\gamma h)\Phi (y^0,w) f (h)
   dh \\
=     &  \int_{G (Z) (F)\lmod G (Y) (\A)}  \sum_{w\in (R_{E/F} (Y\otimes X_0))^+}\omega_{X,Y} (g,h)\Phi (y^0,w) f (h)dh   \\
=& \int_{G (Z) (\A)\lmod G (Y) (\A)} \int_{[G (Z)]}  \sum_{w}\omega_{X,Y} (g, h' h)\Phi (y^0,w) f (h' h)dh' dh \\
=&\int_{G (Z) (\A)\lmod G (Y) (\A)} \int_{[G (Z)]}  f (h'h)\theta_{X_0,Y} (1, h', \omega_{X,Y} (g, h)\Phi (y^0,\cdot)) dh' dh.
\end{align*}
 The inner integral over $h'$ is of the form
\begin{align}\label{eq:inner-integral}
  \int_{[G (Z)]}  f (h')\theta_{X_0,Y} (1, h', \phi) dh'
\end{align}
for some $f\in\sigma$ and $\phi \in \cS ((R_{E/F} (Y\otimes X_0))^+ (\A))$.

By the Main Theorem of \cite{MR1168122}, for any cuspidal automorphic form there  always exists a  non-singular Fourier coefficient that does not vanish. Since we are taking Fourier coefficient of  the first occurrence representation, which is cuspidal,  the Fourier coefficient \eqref{eq:Fourier-coeff-of-theta-lift} is non-vanishing for some choice of data. It follows that the inner integral \eqref{eq:inner-integral} is non-vanishing for some choice of data.

Now we consider the case where $\dim Z = \dim Y - (r-a)$ for $1 \le a \le r$.   Consider theta lift of $\sigma$ from $G (Y)$ to $G (X_{-a})$. This is an earlier lift than the first occurrence, so must vanish. We replace $X$ by $X_{-a}$ in the above computation. If the inner integral is non-vanishing for some choice of data then it is easy to see that  we can choose $\Phi \in \cS (R_{E/F} (Y\otimes \ell_{r-a}^+) (\A)\oplus  (R_{E/F} (  Y\otimes X_0))^+ (\A))$  so that the Fourier coefficient is non-vanishing, which leads to a contradiction. Thus \eqref{eq:key-period-dual} must vanish identically.
\end{proof}

\subsection{$(\chi,b)$-factors and periods}\label{sec:mainperiod}
We are now ready to state and prove the main result of this paper, addressing the relation between
the existence of a $(\chi,b)$-factor in the global Arthur parameter $\psi$ associated to a $\sigma\in\CA_\cusp(G(X))$ and the non-vanishing
of certain period integral of $\sigma$.

\begin{thm}\label{thm:mainperiod}
For $\sigma\in\CA_\cusp(G(X))$, let $\psi_\sigma$ be its global Arthur parameter. If a simple global Arthur parameter $(\chi,b)$, 
with a maximal possible integer $b\geq 1$, occurs
in $\psi_\sigma$ as a simple summand, then there exist a (possibly trivial) anisotropic Hermitian space $Y_0$ with $\dim Y_0 \le \dim X  -b$ and $\dim Y_0 \equiv \epsilon_\chi \mod{2}$, a non-degenerate subspace $Z$ of $X$, with its dimension satisfying
\begin{equation*}
  \frac{\dim X+b+\dim Y_0}{2}\leq \dim Z \le \dim X,
\end{equation*}
 a Bruhat-Schwartz function
  $\phi \in \cS (R_{E/F} (Y_0\otimes X)^+ (\A))$ and a cusp form $f\in \sigma\otimes\chi^{-1}$, such that the period integral
  \begin{equation}\label{eq:period-sigma}
        \int_{[G (Z)]} f (g) \theta_{\psi,\chi_1,\chi_2,X,Y_0} (g,1,\phi) dg
  \end{equation}
converges absolutely and does not vanish.
\end{thm}

\begin{proof}
The condition that a simple global Arthur parameter $(\chi,b)$ occurs
in $\psi_\sigma$ as a simple summand implies that  the partial $L$-function $L^S(s,\sigma\times\chi)$ has a simple pole at $s=\frac{b+1}{2}$ and is holomorphic for $\Re(s)>\frac{b+1}{2}$. By Thm.~\ref{thm:poleL} the lowest occurrence $\LO_{\psi,\chi} (\sigma\otimes\chi^{-1})$ is less than or equal to $\dim X -b$. Assume that the lowest occurrence is realised by the Hermitian space $Y$ of  dimension $m-b'$  with $b'\ge b$ and Witt index $r$. 
Then by Prop.~\ref{prop:period-X->Y}, there exists a non-degenerate subspace $Z$ of $X$ of dimension $\dim X - r$ such that the period integral \eqref{eq:period-sigma} converges absolutely and does not vanish for some choice of cusp form $f\in \sigma\otimes\chi^{-1}$ and Bruhat-Schwartz function  $\phi \in \cS (R_{E/F} (Y_0\otimes X)^+ (\A))$.  The dimension of  $Z$ can be seen to satisfy the following
\begin{align*}
  &\dim Z =\dim X - r = \dim X -\half (\dim Y - \dim Y_0) \\
= &\dim X -\half (\dim X -b' - \dim Y_0) = \half (\dim X +b' + \dim Y_0) \\
\ge&\half (\dim X +b + \dim Y_0).
\end{align*}
This concludes the proof of the theorem.
\end{proof}

It is clear that when $b=0$, the character $\chi$ has no relation with $\sigma$ in the framework of the endoscopic classification of
Arthur (\cite{MR3135650}).
However, the $L$-function $L^S(s,\sigma\times\chi)$ may still carry information about $\sigma$ and $\chi$. In fact, in a special case
treated in Theorem ~\ref{thm:poleL}, the case when $b=0$ is related to the condition that the partial $L$-function
$L^S(s,\sigma\times\chi)$ does not vanish at $s=\frac{1}{2}$. By Prop.~\ref{prop:period-X->Y}, we have the following result, which
complements Thm.~\ref{thm:mainperiod}.
\begin{thm}\label{thm:mainperiod-b=0}
For $\sigma\in\CA_\cusp(G(X))$,  assume that $L^S (s,\sigma \times \chi)$ is non-vanishing at $s=1/2$ and holomorphic for $\Re s > 1/2$ and that   $m-\epsilon_\chi$ is even. Then there exist a (possibly trivial) anisotropic Hermitian space $Y_0$ with $\dim Y_0 \le \dim X $ and $\dim Y_0 \equiv \epsilon_\chi \mod{2}$, a non-degenerate subspace $Z$ of $X$, with its dimension satisfying
\begin{equation*}
  \frac{\dim X + \dim Y_0}{2}\leq \dim Z \le \dim X,
\end{equation*}
 a Bruhat-Schwartz function
  $\phi \in \cS (R_{E/F} (Y_0\otimes X)^+ (\A))$ and a cusp form $f\in \sigma\otimes\chi^{-1}$, such that the period integral
  \begin{equation}\label{eq:period-sigma}
        \int_{[G (Z)]} f (g) \theta_{\psi_F,\chi_1,\chi_2,X,Y_0} (g,1,\phi) dg
  \end{equation}
converges absolutely and does not vanish.
\end{thm}
\begin{proof}
  The argument is similar to that of Thm.~\ref{thm:mainperiod}. By Thm.~\ref{thm:poleL}, the assumptions imply  that the lowest occurrence $\LO_{\psi,\chi} (\sigma\otimes\chi^{-1})$ is less than or equal to $\dim X$. Then the rest of the argument in the proof of Thm.~\ref{thm:mainperiod} goes through with $b$ replaced by $0$.
\end{proof}

\subsection{Certain useful periods}
The next  proposition shows the vanishing of some period integrals over a Jacobi group containing $G (Z)$ where $Z$ is a non-degenerate subspace of $X$ of dimension $\dim X - r$ with $r$ determined by the first occurrence, but not any $G (Z)$ where $Z$ is of dimension greater than $\dim X - r$.
\begin{prop}
  Let $\sigma\in\CA_\cusp(G(X))$ and $Y_0$ a (possibly trivial) anisotropic Hermitian space. Assume that
  \begin{equation*}
    \FO_{\psi,\chi_1,\chi_2,X}^{Y_0} (\sigma) = \dim Y_0+2r
  \end{equation*}
 with $1\le r \le \dim X$. Assume that the first occurrence is realised by $Y$.  Let $Z$ be a non-degenerate subspace of $X$ of dimension $\dim X - (r-2)$  containing  an isotropic vector  $z_0$. Let 
  $R_1^0$ be the subgroup consisting of those elements of $G (Z)$ that fix $z_0$.
Then
\begin{equation*}
  \int_{[R_1^0]} \theta_{\psi,\chi_1,\chi_2,X,Y_0} (g,1,\Phi) f (g) dg
\end{equation*}
is absolutely convergent and vanishes
  for all $\Phi \in \cS ((R_{E/F} (Y_0\otimes X))^+ (\A))$ and $f\in\sigma$.
\end{prop}
We have the  symmetric version:
\begin{prop}
  Let $\sigma\in\CA_\cusp(G(Y))$ and $X_0$ a (possibly trivial) anisotropic skew-Hermitian space. Assume that
  \begin{equation*}
    \FO_{\psi,\chi_1,\chi_2,Y}^{X_0} (\sigma) = \dim X_0+2r
  \end{equation*}
 with $1\le r \le \dim Y$. Assume that the first occurrence is realised by $X$.  Let $Z$ be an non-degenerate subspace of $Y$ of dimension $\dim Y - (r-2)$  containing  an isotropic vector  $z_0$. Let 
  $R_1^0$ be the subgroup of $G (Z)$ consisting of those elements that fix $z_0$.
Then
\begin{equation*}
  \int_{[R_1^0]} \theta_{\psi,\chi_1,\chi_2,X_0,Y} (1,h,\Phi) f (h) dh
\end{equation*}
converges absolutely, and vanishes
  for all $\Phi \in \cS ((R_{E/F} (Y\otimes X_0))^+ (\A))$ and $f\in\sigma$.
\end{prop}

\begin{proof}
  The idea is similar to Prop.~\ref{prop:period-Y->X}. We omit some details which can be found in loc. cit. Let $Z_{-1}$ be the orthogonal complement in $Z$ to a hyperbolic plane containing $z_0$. The absolute convergence follows from Prop.~\ref{prop:period-Y->X} and the fact that $R_1^0$ is a semidirect product of $G (Z_{-1})$ and a unipotent subgroup.

We consider the theta lift of $\sigma$ to $G (W)$ such that $W$ is in the same Witt tower as $X$ with Witt index $r-1$. This is an earlier lift than the first occurrence and therefore must vanish for all choice of data. Let $c$  be a non-degenerate $(r-2)\times (r-2)$-Hermitian matrix. Set $\til {c} = \diag \{c,0\}$. This is a degenerate $(r-1)\times (r-1)$-Hermitian matrix. We take the $\psi_{E,\til {c}}$-th Fourier coefficient of the  theta lift. This is the integral
  \begin{equation*}
    \int_{[N']} \int_{[G (Y)]} \theta_{\psi,\chi_1,\chi_2,W,Y} (n (\beta) g,h,\Phi) f (h)\overline {\psi_{E,\til { c}} (\beta)} dh dn
  \end{equation*}
where $N'$ is the unipotent subgroup of $G (W)$ that is isomorphic to the set of $(r-1)\times (r-1)$-Hermitian matrices. Standard computation shows that the above is equal to
\begin{equation}\label{eq:int-Fourier-c-0}
  \int_{[G (Y)]} \sum_{\substack{y \in Y^{r-1}\\ \form{y}{y}_Y=2\til { c}}}\sum_{u\in R_{E/F} (Y\otimes W_0)^+}\omega_{\psi,\chi_1,\chi_2,W,Y} (g,h) \Phi (y,u) f (h) dh .
\end{equation}
Here the Schwartz space is realised as
\begin{equation*}
  \cS (  R_{E/F} (Y\otimes \ell_{r-1}^-) (\A) \oplus (R_{E/F} (Y\otimes W_0))^+ (\A)).
\end{equation*}
Fix $y^0\in Y^{r-2}$ that represents  $2 c$.  Let $Z$ be the orthogonal complement to $y^0$ in $Y$. Assume that $Z$ is isotropic with an isotropic vector $z_0$. Let $R_1^0$ be the subgroup consisting of those elements that fix $z_0$. There are two orbits of the set of elements in $Y^{r-1}$ that represent $2\til {c}$ under the action of $G (Y)$. The representatives can be chosen to be $(y^0,0)$ and $(y^0,z_0)$.
Then \eqref{eq:int-Fourier-c-0} is equal to
\begin{align*}
  &\int_{[G (Y)]} \sum_{\gamma \in G (Z)\lmod G (Y)}\sum_{u\in (R_{E/F} (Y\otimes W_0))^+}\omega_{W,Y} (g,h) \Phi (\gamma^{-1} ( y^0,0) , u) f (h) dh \\
&+ \int_{[G (Y)]} \sum_{\gamma \in R_1^0 \lmod G (Y)}\sum_{u\in (R_{E/F} (Y\otimes W_0))^+}\omega_{W,Y} (g,h) \Phi (\gamma^{-1} ( y^0,z_0) , u) f (h) dh\\
=&\int_{[G (Y)]} \sum_{\gamma \in G (Z)\lmod G (Y)}\sum_{u\in (R_{E/F} (Y\otimes W_0))^+}\omega_{W,Y} (g,\gamma h) \Phi ((y^0,0) , u) f (h) dh \\
&+ \int_{[G (Y)]} \sum_{\gamma \in R_1^0 \lmod G (Y)}\sum_{u\in (R_{E/F} (Y\otimes W_0))^+}\omega_{W,Y} (g,\gamma h) \Phi ( (y^0,z_0) , u) f (h) dh\\
=&\int_{G (Z) (F)\lmod  G (Y) (\A)} \sum_{u\in R_{E/F} (Y\otimes W_0)^+}\omega_{W,Y} (g, h) \Phi (( y^0,0) , u) f (h) dh \\
&+ \int_{R_1^0 (F)\lmod G (Y) (\A)} \sum_{u\in R_{E/F} (Y\otimes W_0)^+}\omega_{W,Y} (g, h) \Phi ( ( y^0,z_0) , u) f (h) dh.
\end{align*}
Note that $W_0 = X_0$.
The first part has inner integral of the form
\begin{equation*}
  \int_{[ G (Z)] } \theta_{X_0,Y} (1, h', \phi)  f (h') dh'
\end{equation*}
for some $\phi \in \cS((R_{E/F} (Y\otimes X_0))^+(\A))$ and thus must vanish by Prop.~\ref{prop:period-Y->X} because the dimension of $Z$ is too large. Thus the second part must vanish because the sum is known to be zero. This has inner integral of the form
\begin{equation*}
  \int_{[ R_1^0] } \theta_{X_0,Y} (1, h', \phi') f (h') dh'
\end{equation*}
for some $\phi' \in \cS((R_{E/F} (Y\otimes X_0))^+ (\A))$ and $f\in \sigma$. If it does not vanish for some choice of data then it is easy to see that we can choose some data so that the Fourier coefficient does not vanish. Thus the inner integral above must vanish.
\end{proof}

\bibliographystyle{alpha}
\bibliography{MyBib}
\end{document}